\documentclass{article}

\usepackage{geometry}
\usepackage{epsfig}
\usepackage{amssymb}
\usepackage{natbib}

\usepackage[utf8]{inputenc} 
\usepackage[T1]{fontenc}    
\usepackage{hyperref}       
\usepackage{url}            
\usepackage{booktabs}       
\usepackage{amsfonts}       
\usepackage{nicefrac}       
\usepackage{microtype}      
\usepackage{xcolor}         

\usepackage{graphicx}
\usepackage{amsmath}
\usepackage{amsthm}

\usepackage{settings}
\usepackage{authblk}

\author[]{Lucas De Lara\thanks{E-mail: lucas.de-lara@univ-lorraine.fr}}
\affil[]{Institut Elie Cartan de Lorraine, Université de Lorraine}

\title{A clarification on the links between potential outcomes and do-interventions}
\date{}

\begin{document}

\maketitle

\begin{abstract}
Most of the scientific literature on causal modeling considers the structural framework of Pearl and the potential-outcome framework of Rubin to be formally equivalent, and therefore interchangeably uses do-interventions and the potential-outcome framework to define counterfactual outcomes. In this paper, we agnostically superimpose a structural causal model and a Rubin causal model compatible with the same observations to specify under which mathematical conditions counterfactual outcomes obtained via do-interventions and potential outcomes need to, do not need to, can, or cannot be equal (almost surely or in law). Our comparison builds upon the fact that such causal models do not have to produce the same counterfactuals outcomes, and highlights real-world problems where they generally cannot correspond under classical causal-inference assumptions. Then, we examine common claims and practices from the causality literature in the light of this comparison. In doing so, we aim at clarifying the links between the two causal frameworks, and the interpretation of their respective counterfactuals.
\end{abstract}



\section{Introduction}

Understanding causation between phenomena rather than mere association is a fundamental scientific challenge. Over the last decades, two mathematical frameworks using a terminology based on random variables have become the gold standards to address this problem.

On the one hand, the \emph{structural account} of \cite{pearl2009causality} rests on the knowledge of a \emph{structural causal model} (SCM) which specifies all cause-effect equations between observed random variables (often depicted by a graph). The interest of such equations comes from the possibility of carrying out \emph{do-interventions}: forcing a variable to take a given value while keeping the rest of the mechanism untouched. More concretely, let $T$ and $Y$ be observational variables of the model such that we would like to understand the downstream effect of $T$ onto $Y$. Replacing the formula generating $T$ by $T = t$ for a given possible value $t \in \T$ and propagating this change through the other equations defines the altered variable $Y_{T=t}$, representing $Y$ \emph{had $T$ been equal to $t$}.

On the other hand, the \emph{potential-outcome account} of \cite{rubin1974estimating} mathematically formalizes causal inference in clinical trials. Letting $T$ denote a \emph{treatment status} (\emph{e.g.}, taking a drug or not) and $Y$ an \emph{outcome} of interest (\emph{e.g.}, recovering or not), a so-called \emph{Rubin causal model} (RCM) postulates the existence of \emph{potential outcomes} $(Y_t)_{t \in \T}$ representing what the outcome would be \emph{were $T$ equal to $t$} for any $t \in \T$. The \emph{fundamental problem of causal inference} \citep{holland1986statistics} refers to the fact that in practice we cannot observe simultaneously all the potential outcomes, rendering unidentifiable the causal effect of $T$ onto $Y$. Nevertheless, causal inference can still be achieved thanks to a mix of untestable assumptions and statistical tools: adjusting on a set of available covariates $X$ containing all possible \emph{confounders} between the treatment and the potential outcomes permits to identify the law of counterfactual outcomes.

Each of these causal theories enables one to carry out \emph{counterfactual reasoning}, that is answering contrary-to-fact questions such as \say{Had they taken the drug, would have they recovered?}: by applying do-interventions on an SCM, one can compute the outcome $Y_{T=t}$ for every possible treatment status $t \in \T$; using an RCM with appropriate hypotheses, one can infer the law of the every potential outcome $Y_t$. Both approaches involve variables describing counterfactual outcomes, more precisely outcomes \emph{had the variable $T$ taken a certain value}. This naturally raises the question: are these outcome variables equal (almost surely or in law) across frameworks? A plethora of scientific books and survey papers interchangeably use Pearl's do notation and the potential-outcome subscript notation to write outcomes after interventions, suggesting that the corresponding definitions of counterfactuals are identical and differ only from theirs perspectives \citep{imbens2020potential, neal2020introduction, barocas-hardt-narayanan, colnet2024causal, makhlouf2024causality}. To justify this, they often refer to Pearl, who argued that \say{the two frameworks can be used interchangeably and symbiotically}.\footnote{\url{http://causality.cs.ucla.edu/blog/index.php/2012/12/03/judea-pearl-on-potential-outcomes/}} However, influential works on \say{equivalences} between the two causal frameworks have mostly focused on translating graphical assumptions into conditional-independence restrictions instead of actually proving whether counterfactual outcomes were equal across models, or implicitly addressed specific cases. Notably, \cite[Chapter 7]{pearl2009causality} and \citep{richardson2013single} consider \emph{ex nihilo} the exchangeability of the two associated notations in their unifications of both causal frameworks.

In this paper, we essentially aim at clarifying in which sense using interchangeably two distinct causal models is appropriate. To this end, we compare a potential-outcome model and a structural causal model compatible with a same distribution of observations from an agnostic perspective. We introduce three levels of comparisons, corresponding to different degrees of counterfactual reasoning, and neutrally ask under which conditions two models are (un)distinguishable at these levels. This analysis crucially reminds that the models are \emph{not} mathematically bound to correspond, meaning that using them symbiotically generally rests on a \emph{choice}. Moreover, it classifies real-world scenarios where the models can(not) be exchanged, depending on their respective assumptions. Then, we interpret the counterfactual statements and causal effects respectively induced by $(Y_t)_{t \in \T}$ and $(Y_{T=t})_{t \in \T}$ when the models do not coincide, and explain how such results relate to the so-called \emph{formal equivalence} between causal frameworks accepted by the causal-inference community.

In a similar vein,  \cite{ibeling2024comparing} recently provided an in-depth theoretical comparison of the two frameworks by adopting a neutral viewpoint, notably proving that a well-behaved RCM can always be \emph{represented} by an SCM. Altogether, our contributions supplement their work by specifying graphical assumptions that must generally satisfy such an SCM, and by giving a real-world interpretation to these assumptions. In particular, analyzing theoretical representability results through the prism of practically relevant problems enables us to point out overlooked paradoxes in the causal-inference literature. On the basis of our results and discussions, we call the community to rigorously justify their exchanges of models across frameworks, as it could lead to misleading conclusions. In doing so, we hope to further clarify the role of each causal modeling in the past, current, and future causal-inference research.

\subsection{Motivating example}

To motivate this work, we illustrate on a concrete example how potential outcomes and structural counterfactuals can be respectively used to address a specific causal-inference problem.

\subsubsection{Problem}\label{sec:ex_intro}

For simplicity and concision, this example uses informal definitions of causal models; Sections~\ref{sec:prelim} and \ref{sec:setup} introduce formal definitions. All generic notations are specified in Section~\ref{sec:notations}. 

We consider the following fairness-inspired problem. Let the \emph{treatment status} $T : \Omega \to \{0,1\}$ be a binary random variable indicating the gender, $T(\omega) = 0$ standing for women and $T(\omega) = 1$ standing for men; let the \emph{covariate} $X : \Omega \to \R$ be a random variable quantifying the level of work experience, a higher score encoding a more adapted experience; let the \emph{outcome} $Y : \Omega \to \R$ be a random variable evaluating a candidate's application for some position, a better score giving a higher probability of acceptance. Causal analysts are tasked with assessing the fairness of the evaluation process, by answering the question \say{What is the average outcome variation for individuals with a given profile had their gender changed?}. To do so, analyst $\M$ relies on a structural causal model to compute the do-intervention outputs  $(Y_{T=0},Y_{T=1})$ (called structural counterfactuals), and then estimates $\E[Y_{T=1} - Y_{T=0} \mid X=x]$. For their part, analyst $\mathcal{R}$ introduces potential outcomes $(Y_0,Y_1)$ satisfying causal-inference assumptions, and then estimates $\E[Y_1 - Y_0 \mid X=x]$. The signs and intensities of the above estimands quantify (un)fairness. We ask whether analysts $\M$ and $\mathcal{R}$ reach the same conclusion, in other words if the two causal estimands are equal.

Analyst $\M$ postulates that $(T,X,Y)$ is ruled by the following collection of structural assignments:
\begin{align*}
    T &= U_T,\\
    X &= \alpha T + U_X,\\
    Y &= X + \beta T + U_Y,
\end{align*}
where $\alpha$ and $\beta$ are deterministic parameters quantifying the causal influence of $T$ onto respectively $X$ and $Y$, and $U_X$ represents the hidden merit or effort of an individual. Typically, a positive parameter $\alpha$ describes the societal inequalities leading women to have a lower level of work experience than men with equal merit $U_X$. The variables $U_T$, $U_X$, and $U_Y$ are exogenous noises such that $U_Y \independent (T,X)$. For every $t \in \{0,1\}$, the do-intervention denoted by $\operatorname{do}(T=t)$ enables one to compute the downstream effect onto $X$ and $Y$ of fixing $T$ to the value $t$. Concretely, it recursively defines $X_{T=t} = \alpha t + U_X$ and $Y_{T=t} = X_{T=t} + \beta t + U_Y = (\alpha + \beta) t + U_X + U_Y$. Therefore,
\[
    \E[Y_{T=1} - Y_{T=0} \mid X=x] = \alpha + \beta.
\]

Analyst $\mathcal{R}$ postulates two potential outcomes $Y_0$ and $Y_1$, relating to the factual outcome $Y$ via the following consistency property: $Y = (1-T) \cdot Y_0 + T \cdot Y_1$. To compute their estimand from observational data, they suppose that $\P(T=1 \mid X=x) > 0$ and $(Y_0,Y_1) \independent T \mid X$. These conditions are referred to as the fundamental assumptions of causal inference. They imply that $\E[Y_1 - Y_0 \mid X=x] = \E[Y \mid X=x, T=1] - \E[Y \mid X=x, T=0]$ (see Lemma~\ref{lm:sw}). We emphasize that, while they leverage different techniques, the two analysts work with the same $(T,X,Y)$. Therefore, $\E[Y \mid X=x, T=t] = x + \beta t + \E[U_Y]$ due to $U_Y \independent (T,X)$. This leads to
\[
\E[Y_1 - Y_0 \mid X=x] = \beta.
\]

Consequently, $\E[Y_{T=1} - Y_{T=0} \mid X=x] \neq \E[Y_1 - Y_0 \mid X=x]$ if $\alpha \neq 0$. This means that the two analysts obtain distinct, possibly contradictory results.

\subsubsection{Comments}

This example shows that potential outcomes $(Y_t)_{t \in \T}$ equipped with the fundamental assumptions of causal inference do not necessarily define the same causal effects as the post-intervention outcomes $(Y_{T=t})_{t \in \T}$ obtained via do-interventions on $T$. As such, one cannot always interchangeably employ given potential outcomes and structural counterfactuals. This illustration motivates a general comparison of the two causal-inference approaches, which is precisely the goal of this paper. Let us summarize how the example's specifics connect to the broader results presented in this manuscript.

In Section~\ref{sec:setup}, we propose to analyze the similarities and differences of the two approaches through the prism of equivalence relations between causal models. These relations characterize which causal estimands coincide across two models. For instance, in the above example $\E[Y_{T=1} - Y_{T=0} \mid X=x] \neq \E[Y_1 - Y_0 \mid X=x]$ means that the corresponding SCM and RCM were not \emph{single-outcome equivalent} (Definition~\ref{def:representation}).

In Section~\ref{sec:main}, we notably investigate under which conditions an RCM satisfying the fundamental assumptions of causal inference and an SCM are equivalent or not. More precisely, we show that they are generally single-outcome equivalent just in case $T$ does not cause $X$ (Theorem~\ref{thm:identification}). Therefore, in typical fairness problems where $T$ is an immutable variable like sex or race, swapping potential outcomes and structural counterfactuals in a mathematical formula commonly yields distinct estimands. This explains the results of the above illustration. Our analysis also provides a general interpretation of the respective counterfactual variables: a structural counterfactual $Y_{T=t}$ of $Y$ changes by do-intervention $T$ into $t$ and $X$ into $X_{T=t}$ while keeping $U_Y$ equal; a potential outcome $Y_t$ of $Y$ under the assumptions of causal inference changes $T$ into $t$ while keeping $X$ and $U_Y$ equal. This difference in meanings justifies why $\E[Y_{T=1} - Y_{T=0} \mid X=x]$ captured the total effect $\alpha+\beta$ of $T$ on $Y$, whereas $\E[Y_1 - Y_0 \mid X=x]$ captured the direct effect $\beta$.

In Section~\ref{sec:consequences}, we detail how these results relate to the formal equivalence between causal frameworks. Our discussion delineates two paradigms for defining and using potential-outcome variables. The first is what analyst $\M$ does: defining potential outcomes via do-interventions on the latent SCM, so that the properties they satisfy stem from the SCM. The second is what analyst $\mathcal{R}$ does: defining potential outcomes with a different causal interpretation than the one given by do-interventions, in their case by directly placing assumptions upon them---regardless of what holds in the latent SCM. Crucially, only analyst $\M$ relies on the formal equivalence. This is why we caution against unjustified exchanges of the potential-outcome subscript notation and the do notation.

\subsection{Outline of the paper}

The rest of the paper proceeds as follows. Section~\ref{sec:prelim} furnishes the basic knowledge on structural causal models and potential-outcome models. Section~\ref{sec:setup} formalizes the problem we address by introducing notions of equivalence between causal models. Section~\ref{sec:main} clarifies under which conditions a structural causal model and a potential-outcome model compatible with the same observational data need or need not be equivalent. Section~\ref{sec:consequences} discusses the relation between these results and the formal equivalence between causal frameworks. Appendix~\ref{sec:middle_case} addresses a supplementary illustration. Proofs of intermediary results are deferred to Appendix~\ref{sec:proofs}.

\section{Preliminaries}\label{sec:prelim}

This section provides the necessary background on structural causal models and potential outcomes. It is meant to keep the paper self-contained. Section~\ref{sec:notations} introduces generic mathematical notations; Section~\ref{sec:pearl} presents Pearl's causal framework; Section~\ref{sec:rubin} explains Rubin's causal framework.

\subsection{Basic mathematical notations}\label{sec:notations}

Throughout, we consider a probability space $(\Omega, \Sigma, \P)$ with $\Omega$ a sample space, $\Sigma$ a $\sigma$-algebra, and $\P : \Sigma \to [0,1]$ a probability measure. This space does not necessarily have a physical interpretation; it abstractly represents the possible underlying states of the world. Crucially, it serves as the common mathematical basis to define and compare random variables.

A \emph{random variable} $W$ (including \emph{random vectors}) is a measurable function from $\Omega$ to a Borel subset of an Euclidean space equipped with the Borel $\sigma$-algebra. It produces a probability distribution on its output space: we write $\mathcal{L}(W) := \P \circ W^{-1}$ and $\E[W] := \int W(\omega) \mathrm{d}\P(\omega)$ for respectively the \emph{law and expectation under $\P$} of a random variable $W$.\footnote{To fix ideas, one can consider that $\Omega = [0,1]$ and that $\P$ is the uniform distribution on $\Omega$ so that for any Borel probability distribution $P$ there exists $W$ such that $P = \P \circ W^{-1}$ (as reminded in the proof of Proposition~\ref{prop:existence}). As such, this choice allows one to define random variables or vectors with any laws.} We emphasize that the laws of univariate random variables can be completely general in this paper; we do not suppose them to be either Lebesgue-absolutely continuous or discrete. For any Borel set $F$, we use the common probability-textbook notation $\{W \in F\}$ for the set $\{\omega \in \Omega \mid W(\omega) \in F\} \in \Sigma$. Two variables $W_1$ and $W_2$ are \emph{$\P$-almost surely equal}, denoted by $W_1 \aseq W_2$, if $\P(W_1=W_2)=1$; they are \emph{equal in law under $\P$}, denoted by $\mathcal{L}(W_1) = \mathcal{L}(W_2)$, if $\P(W_1 \in F) = \P(W_2 \in F)$ for every Borel set $F$. The notation $W_1 \independent W_2$ means that $W_1$ and $W_2$ are \emph{independent under $\P$}, that is $\P(W_1 \in F_1, W_2 \in F_2) = \P(W_1 \in F_1) \cdot \P(W_2 \in F_2)$ for all Borel sets $F_1,F_2$.

We denote by $\P( \cdot \mid W=w)$ the \emph{regular conditional probability measure with respect to $\{W=w\}$}, which exists and is unique for $\law{W}$-almost every $w$. Then, whenever they are well-defined, we write $\mathcal{L}(W_2 \mid W_1=w_1) := \P(\cdot \mid W_1=w_1) \circ {W_2}^{-1}$ and $\E[W_2 \mid W_1=w_1] := \int W_2(\omega) \mathrm{d}\P(\omega \mid W_1=w_1)$ for respectively the \emph{law and expectation of $W_2$ conditional to $W_1=w_1$}. The expression $W_1 \independent W_2 \mid W_3$ means that $W_1$ and $W_2$ are \emph{independent conditional to $W_3$ under $\P$}, namely that $W_1$ and $W_2$ are independent under $\P(\cdot \mid W_3=w_3)$ for $\law{W_3}$-almost every $w_3$.

Moreover, for any tuple $w := (w_i)_{i \in \I}$ indexed by a finite index set $\I$ and any subset $I \subseteq \I$ we write $w_I := (w_i)_{i \in I}$. Then, for any index set $J$ we abusively write $w_J$ for $w_{J \cap \I}$ and use the notation $w_J = \emptyset$ whenever $J \cap \I = \emptyset$. Similarly, we define the Cartesian product $\mathcal{W}_I := \prod_{i \in I} \mathcal{W}_i$ for any collection of spaces $(\mathcal{W}_i)_{i \in \I}$.

\subsection{Pearl's causal framework}\label{sec:pearl}

Pearl's causal modeling mathematically formalizes associations that standard probability calculus cannot describe through the notions of structural causal models and do-interventions \citep{pearl2009causality}. This section recalls the basics on this topic, borrowing the introduction proposed in \citep{blom2020beyond,bongers2021foundations}.

\subsubsection{Structural causal models}

A \textit{structural causal model} (SCM) represents the causal relationships between the studied variables. It is the cornerstone of Pearl's causal framework.
\begin{definition}[Structural causal model]\label{def:scm}
Let $\I$ and $\J$ be two disjoint finite index sets, and write $\mathcal{V} := \prod_{i \in \I} \mathcal{V}_i \subseteq \R^{|\I|}$, $\mathcal{U} := \prod_{j \in \J} \mathcal{U}_j \subseteq \R^{|\J|}$ for two Borel product spaces. A \emph{structural causal model} $\mathcal{M}$ is a couple $\langle U, g \rangle$ where:

\begin{enumerate}
    \item $U := (U_j)_{j \in \J}$ is a collection of mutually independent random variables called the \emph{random noises}, such that $U_j$ is from $\Omega$ to $\mathcal{U}_j$ for every $j \in \J$;
    \item $g := (g_i)_{i \in \I}$ is a collection of measurable $\R$-valued functions, where for every $i \in \I$ there exist two subsets of indices $\operatorname{Endo}(i) \subseteq \I$ and $\operatorname{Exo}(i) \subseteq \J$, respectively called the \emph{endogenous} and \emph{exogenous parents} of $i$, such that $g_i$ is from $\mathcal{V}_{\operatorname{Endo}(i)} \times \mathcal{U}_{\operatorname{Exo}(i)}$ to $\mathcal{V}_i$.\footnote{This definition tolerates that distinct endogenous variables share the same exogenous parents, that is $\operatorname{Exo}(i) \cap \operatorname{Exo}(i') \neq \emptyset$ for some $i \neq i'$. Therefore, the variables in $(U_{\operatorname{Exo}(i)})_{i \in \I}$ are not necessarily mutually independent.}
\end{enumerate}
A random vector $V : \Omega \to \mathcal{V}$ is a solution of $\mathcal{M}$ if for every $i \in \I$,
\begin{equation}\label{eq:causal_eq}
    V_i \aseq g_i(V_{{\text{Endo}}(i)},U_{\text{Exo}(i)}).
\end{equation}
The equations defined by \eqref{eq:causal_eq} and characterized by $g$ and $U$ are called the \emph{structural equations}.\footnote{In the paper, we will often informally define an SCM by specifying the structural equations rather than $g$.}
\end{definition}

Such a model explains how some \textit{endogenous} variables $V$, representing observed data, are generated from \textit{exogenous} variables $U$, describing background factors. The structural equations quantify the causal dependencies between all these variables and are frequently illustrated by the directed graph $\G_\M$ with nodes $\I \cup \J$, and such that a directed edge points from node $k$ to node $l$ if and only if $k \in \text{Endo}(l) \cup \text{Exo}(l)$ (we say in this case that $k$ is a parent of $l$). For convenience, we make the common assumption that the studied models are \emph{acyclic}, which means that their associated graphs do not contain any cycles.
\begin{assumption}[Acyclicity]\label{hyp:acyclic} $\mathcal{M}$ is such that $\G_\M$ is a directed \emph{acyclic} graph.   
\end{assumption}
Not only acyclicity simplifies the interpretation of causal dependencies, but it entails \emph{unique solvability} of the SCM: according to \cite[Proposition 3.4]{bongers2021foundations}, Equation~$\eqref{eq:causal_eq}$ admits a unique solution up to $\P$-negligible sets. We will abusively refer to such a solution as \emph{the} solution of the SCM. Also, on the basis of this well-posedness of the solution $V$, we will often replace the indices $i \in \I$ and $j \in \J$ by the associated random variables $V_i$ and $U_j$ in the relevant notations. For example we may write $\operatorname{Endo}(V_i)$ instead of $\operatorname{Endo}(i)$.

The purpose of causal structures is to capture the assumption that variables are not independently manipulable. As we detail next, they enable to understand the downstream effect of fixing some variables to certain values onto nonintervened variables.

\subsubsection{The do-intervention}

A \emph{perfect do-intervention} is an operation forcing a set of endogenous variables to take predefined values while keeping all the rest of the causal mechanism equal.

\begin{definition}[Perfect do-intervention]\label{def:dointervention} Let $\mathcal{M} = \langle U, g \rangle$ be an SCM, $I \subseteq \I$ a subset of endogenous variables, and $\tilde{v}_I \in \mathcal{V}_I$ a value. The action $\operatorname{do}(I,\tilde{v}_I)$ defines the modified model $\mathcal{M}_{\operatorname{do}(I,\tilde{v}_I)} = \langle U, \tilde{g} \rangle$ where $\tilde{g}$ is given by: for any $(v,u) \in \V \times \U$ and $i \in \I$,
$$
    \tilde{g_i}(v_{\operatorname{Endo}(i)},u_{\operatorname{Exo}(i)}) := \begin{cases}
                    \tilde{v}_i \text{ if } i \in I,\\
                    g_i(v_{\operatorname{Endo}(i)},u_{\operatorname{Exo}(i)}) \text{ if } i \in \I \setminus I.
                   \end{cases}  
$$
\end{definition}

A do-intervention preserves acyclicity, and therefore unique solvability. As a consequence, if $V$ is the solution of an acyclic $\M$, one can define (up to $\P$-negligible sets) its post-intervention counterpart $V_{\operatorname{do}(I,v_I)}$ solution to $\M_{\operatorname{do}(I,v_I)}$. It describes an alternative world where every $V_i$ for $i \in I$ is set to value $v_i$ while $U$ is preserved. In the rest of the paper, we simply write $\operatorname{do}(V_I = v_I)$ for the operation $\operatorname{do}(I,v_I)$, and use the subscript $V_I=v_I$ to indicate results of this operation. For instance, we write $\M_{V_I=v_I}$ for $\M_{\operatorname{do}(I,v_I)}$ and $V_{V_I=v_I}$ for $V_{\operatorname{do}(I,v_I)}$. Crucially, intervening does not amount to conditioning in general, that is $\mathcal{L}(V \mid V_I=v_I) \neq \mathcal{L}(V_{V_I=v_I})$.

The next proposition provides a general expression of the solution before and after intervention, and will play a key role throughout this paper. For any $I \subseteq \I$ we write $\operatorname{Endo}(I) = \cup_{i \in I} \operatorname{Endo}(i)$, $\operatorname{Exo}(I) = \cup_{i \in I} \operatorname{Exo}(i)$, and $I^c = \I \setminus I$.
\begin{lemma}[do-intervention on variables]\label{lem:docalculus} Let $\mathcal{M} = \langle U, g \rangle$ be an SCM satisfying acyclicity (Assumption~\ref{hyp:acyclic}) with solution $V$, and consider $I \subseteq\I$. There exists a deterministic measurable function $f_{I^c}$ such that
\[
    V_{I^c} \aseq f_{I^c}(V_{\operatorname{Endo}(I^c) \setminus I^c}, U_{\operatorname{Exo}(I^c)}).
\]
Moreover, for any intervention $\operatorname{do}(V_I = v_I)$ the solution $\tilde{V}$ of $\mathcal{M}_{V_I = v_I}$ satisfies
\begin{align*}
    &\tilde{V}_{I^c} \aseq f_{I^c}(v_{\operatorname{Endo}(I^c) \setminus I^c}, U_{\operatorname{Exo}(I^c)}),\\
     &\tilde{V}_I \aseq v_I.   
\end{align*}
\end{lemma}

Importantly, this is the same deterministic function $f_{I^c}$ and the same random noises $U_{\operatorname{Exo}(I^c)}$ that generate $V_{I^c}$ and its post-intervention counterpart $\tilde{V}_{I^c}$, the only change being the assignment $V_I = v_I$. Slightly abusing notations, we will often artificially extend the input variables of $f_{I^c}$ to write $V_{I^c} \aseq f_{I^c}(V_I,U_{\operatorname{Exo}({I^c})})$ and $\tilde{V}_{I^c} \aseq f_{I^c}(v_I,U_{\operatorname{Exo}({I^c})})$ using the fact that $\operatorname{Endo}({I^c}) \setminus {I^c} \subseteq I$. Lemma~\ref{lem:docalculus} can be seen as a \say{vectorization} of the structural equations.

\subsubsection{Counterfactual inference with structural causal models}\label{sec:3steps}

Counterfactual inference aims at predicting outcomes had a certain event occurred given some factual observations. Typically, it addresses what-if questions such as \say{Had they been a woman, would have they gotten the position?}. Perfect interventions combined with conditioning provides a natural probabilistic framework to address counterfactual queries. Let for instance $V := (T,X,Y)$ be the solution to an acyclical SCM $\M := \langle U, g \rangle$. Pearl answers the question \say{had $T$ been equal to $t$, what would have been the law of $Y$ given the factual context $X=x$?} by using the so-called \emph{three-step procedure} \citep{pearl2009causality}:
\begin{enumerate}
    \item \textbf{(Abduction)} compute $\mathcal{L}(U \mid X=x)$, the posterior distribution of $U$-values compatible with the context $\{X=x\}$;
    \item \textbf{(Action)} carry out a do-intervention on $\M$ to obtain the intervened causal mechanism $g_{T=t}$ of $\M_{T=t}$;
    \item \textbf{(Prediction)} pass the posterior distribution $\mathcal{L}(U \mid X=x)$ through $g_{T=t}$ to generate the distribution $\mathcal{L}(Y_{T=t} \mid X = x)$ of counterfactual outcomes.
\end{enumerate}
More generally, an SCM enables one to sample from probability distributions of counterfactual outcomes for any choices of context, variables to alter by do-intervention, and outcomes of interest. Note that, while a fully specified SCM permits to compute any counterfactual distributions via the above algorithm, some distributions can also be identified from only the observational distribution $\law{V}$ and the causal graph $\G_\M$ (that is without completely knowing $g$ or the law of $U$) under specific graphical assumptions (see for example the counterfactual intrepetation of the \emph{backdoor criterion} \citep[Theorem 4.3.1]{pearl2016causal}). Before turning to the potential-outcome framework, let us precise our notations and definitions related to the \say{do}.

\begin{remark}[do-interventions versus the do-operator]
Some readers may be familiar with the \emph{do-operator} and the associated notations \citep[Equation 1.37]{pearl2009causality}, which do not exactly correspond to the do-intervention operation from Definition~\ref{def:dointervention}. This operator serves to distinguish between probability distributions of endogenous variables after distinct do-interventions, similarly to how conditioning distinguishes between probability distributions given distinct realizations of random variables. The respective notations of the do-operator and the conditioning operation are similar (but the do-operator respects different rules). Typically, one writes $\E[V \mid \operatorname{do}(V_I=v_I)]$ for the expectation of the solution of $\M_{V_I=v_I}$, and $\E[Y \mid \operatorname{do}(T=t), X=x]$ for the conditional expectation in $\M_{T=t}$ of the outcome given the covariates. These quantities can be expressed with the subscript-variable notations we introduced as follows: $\E[V \mid \operatorname{do}(V_I=v_I)] = \E[V_{V_I=v_I}]$ and $\E[Y \mid \operatorname{do}(T=t), X=x] = \E[Y_{T=t} \mid X_{T=t}=x]$. We refer to  \citep[Chapter 4]{pearl2016causal} for more details.

The do-operator along with the rules of do-calculus enable one to reframe post-intervention probabilities in terms of standard (conditional) probabilities \citep[Theorem 3.4.1]{pearl2009causality}, making them useful for conducting \emph{interventional} inference. However, as explained by Pearl \citep[Section 4.1]{pearl2016causal}, they do not fully allow \emph{counterfactual} inference. Notably, the do-operator distinguishes between probability distributions entailed by distinct interventions but not between random variables entailed by distinct interventions, and thereby fails to represent cross-world dependencies. For example, a quantity like $\E[Y_{T=t} \mid X=x]$, which involves variables defined across $\M_{T=t}$ and $\M$, cannot always be captured by the do-operator. In particular, $\E[Y_{T=t} \mid X=x] \neq \E[Y \mid \operatorname{do}(T=t), X=x]$ in general \citep[Equation 4.8]{pearl2016causal}.

All in all, the do-operator does not provide a sufficient vocabulary for the problem we address: determining under which conditions one can substitute in causal estimands the counterfactual \emph{variables} $(Y_t)_{t \in \T}$ of Rubin's framework with the post-intervention \emph{variables} $(Y_{T=t})_{t \in \T}$ of Pearl's framework. This is why we do not to use notations as $\E[\ \cdot \mid \operatorname{do}(X=x)]$ in this paper. We see a do-intervention as a transformation of a random variable (instead of a distribution), and indicate post-intervention variables by a subscript as commonly done in the structural-counterfactual literature \citep[Chapter 4]{pearl2016causal}. Moreover, we only employ classical operators from probability theory, like $\E[\ \cdot \mid X=x]$ (as defined in Section~\ref{sec:notations}).
\end{remark}

\subsection{Rubin's causal framework}\label{sec:rubin}

The potential-outcome framework, also known as \emph{Neyman-Rubin causal modeling} \citep{neyman1923applications,rubin1974estimating}, was designed to understand the causal effect of a treatment onto an outcome of interest, for instance when one aims at assessing the contribution of a drug to recovering from some disease in clinical trials. In this section, we introduce this framework in the specific case of a binary treatment.

\subsubsection{Potential outcomes}\label{sec:po}

Let $T : \Omega \to \{0,1\}$ represent a binary \emph{treatment status}, typically such that $T(\omega)=0$ indicates the absence of treatment and $T(\omega)=1$ indicates a treatment. More generally, it can encode any distinction between some groups (\emph{e.g.}, men and women). Assuming \emph{no interference between units},\footnote{Paraphrasing Rubin \cite{rubin2010causal}, a \emph{unit} refers to a study object (like a person). This assumption excludes cases where the treatment of one unit may affect the outcome of another.} this framework postulates two \emph{potential outcomes} $Y_0 : \Omega \to \R$ and $Y_1 : \Omega \to \R$, one for each treatment status. These potential outcomes as well as the treatment may depend on some covariates $X : \Omega \to \R^d$ (such as the patient's weight, height, or historical data in clinical trials). Critically, we cannot observe simultaneously $Y_0(\omega)$ and $Y_1(\omega)$ for a same $\omega$: a problem referred as the \emph{fundamental problem of causal inference} \citep{holland1986statistics}. We only have access to the realized \emph{outcome variable} $Y : \Omega \to \R$ which is supposed to be \emph{consistent} with $(Y_0,Y_1)$, that is satisfying $Y = (1-T) \cdot Y_0 + T \cdot Y_1$. Concretely, if $T(\omega)=1$ for some $\omega \in \Omega$, then $Y(\omega)=Y_1(\omega)$, and $Y_0(\omega)$ becomes unidentifiable by mere observations. We refer to the random vector $(T,X,Y_0,Y_1)$ as the \emph{Rubin causal model} (RCM), which is an augmented version of $(T,X,Y)$ due to consistency.

Understanding the causal relationship between the treatment and the outcome in this framework ideally consists in answering counterfactual questions such as \say{What would have been the value of $Y(\omega)$ had $T(\omega)$ been equal to 1 instead of 0 for a specific $\omega$ (such that $X(\omega)=x$)?}. This cannot be answered since the value of either $Y_0(\omega)$ or $Y_1(\omega)$ will always be missing. Instead, in practice, one estimates and compares under some assumptions features of $\mathcal{L}(Y_1)$ and $\mathcal{L}(Y_0)$, or $\mathcal{L}(Y_1 \mid X=x)$ and $\mathcal{L}(Y_0 \mid X=x)$. People commonly focus on computing the \emph{average treatment effect} $\E[Y_1-Y_0]$ or the \emph{conditional average treatment effect} (CATE) $\E[Y_1-Y_0 \mid X=x]$. The main challenge lies in the fact that \emph{association is not causation} in general. In particular, the quantity $\E[Y \mid T=t]$ does not necessarily coincide with the quantity $\E[Y_t]$ for $t \in \{0,1\}$. Typically, if some medical treatment is more likely to be taken by weaker patients, we may observe a lower rate of recovery among the treated group compared to the nontreated group due to the health condition even though the medicine does increase recovery all other things being kept equal: we would observe $\E[Y \mid T=1] < \E[Y \mid T=0]$ while $\E[Y_1] > \E[Y_0]$ (a phenomenon that can be seen as a consequence of \emph{Simpson's paradox} \citep{blyth1972simpson}). In this case, the health condition is called a \emph{confounder}: a variable associated with both the distribution of the treatment and the outcome. However, causal inference from observational data is still possible, as explained next.
 
\subsubsection{Counterfactual inference with fundamental assumptions}\label{sec:causal_estimation}

Estimating a feature of $\law{(T,X,Y_0,Y_1)}$, like a treatment effect, requires expressing it in terms of features of $\law{(T,X,Y)}$ which generates the empirical observations. Such \emph{identifications} can be achieved under two fundamental assumptions. The first one goes by many names through the literature: \emph{conditional ignorability}, \emph{conditional exchangeability}, \emph{conditional exogeneity}, and \emph{conditional unconfoundedness} (among others). Originally formulated by \cite{rosenbaum1983central}, it states that the potential outcomes are independent of the treatment conditional to the covariates, that is $(Y_0,Y_1) \independent T \mid X$. Said differently, it ensures that all confounders between the treatment and the potential outcomes are included in the covariates. Note that this assumption is untestable, as it would require to observe simultaneously the two potential outcomes. The second key hypothesis is \emph{positivity}, which ensures that all units can be exposed to both treatment statuses, that is $0 < \P(T=1 \mid X=x) < 1$ for $\law{X}$-almost every $x \in \R^d$. It readily follows from positivity that the probability distribution $\mathcal{L}(Y \mid X=x,T=t)$ is well defined for $\mathcal{L}(X)$-almost every $x \in \R^d$ and every $t \in \{0,1\}$, and from conditional ignorability that it coincides with $\mathcal{L}(Y_t \mid X=x)$, meaning that association-based outcomes have a causal interpretation. Several statistical methods coexist to estimate the (conditional) average causal effect, all building upon this implication (see for instance \citep{imbens2004nonparametric, yao2021survey}). We do not detail them for concision and clarity since it is not the topic of this paper. We only point out that, similarly to SCMs, the potential-outcome framework enables one to infer distributions of counterfactual outcomes.

\section{Problem setup}\label{sec:setup}

This section precises the problem we address: analyzing the mathematical similarities and differences between two models respectively derived from the two causal frameworks. Section~\ref{sec:models} introduces a potential-outcome model and a structural causal model compatible with a same dataset, and formalizes the assumptions we may place upon them. Section~\ref{sec:equivalence} defines notions of equivalence between the two models, corresponding to different levels of comparison.

\subsection{Causal models and assumptions}\label{sec:models}

The next definition formalizes the notion of observational data in a causal-inference experiment. It generalizes Section~\ref{sec:prelim} by considering a nonbinary treatment and a multivariate outcome.

\begin{definition}[Observational vector]\label{def:obs}
Let $N, d, p \geq 1$ be integers, and $\T := \{0,1,\ldots,N\}$. An \emph{observational vector} $\O$ is a random vector $(T,X,Y)$ where $T : \Omega  \to \T$ is the \emph{treatment status}, $X : \Omega \to \R^d$ the \emph{covariates}, $Y : \Omega \to \R^p$ the \emph{outcome}, and such that $0 < \P(T=t)$ for all $t \in \T$.    
\end{definition}

In order to compare the two causal frameworks, we consider in what follows a superimposed construction where an observational vector $\O := (T,X,Y)$ is concurrently governed by an RCM and an SCM. Figure~\ref{fig:principle} illustrates this construction. We emphasize that we adopt an agnostic approach where there is no presumed relation between the two causal models except their compatibility with $\O$. This is meant to highlight when equality (possibly in law) between potential and structural counterfactual outcomes is a mathematical necessity or the result of specific assumptions.

\subsubsection{Potential-outcome model}

On the one hand, we assume that $T$ is the treatment status, $X$ some covariates, and $Y$ the outcome of interest in a potential-outcome model in the sense of the definition below. 

\begin{definition}[Rubin causal model]\label{def:rcm}
Let $\O := (T,X,Y)$ be an observational vector. A \emph{Rubin causal model (or potential-outcome model) compatible with $\O$} is a random vector $\mathcal{R} := (T,X,(Y_t)_{t \in \T})$ such that the tuple $(Y_t)_{t \in \T} : \Omega \to (\R^p)^{N+1}$ meets the \emph{consistency rule}:
\[
Y \aseq \sum_{t \in \T} \mathbf{1}_{\{T=t\}} Y_t.
\]
For any $t \in \T$, $Y_t$ is referred to as the \emph{potential outcome} had the treatment been equal to $t$. We denote by $\mathfrak{R}_\O$ the class of Rubin causal models compatible with $\O$.
\end{definition}

A generic RCM $\mathcal{R} := (T,X,(Y_t)_{t \in \T}) \in \mathfrak{R}_{\O}$ is basically an augmented version of $\O$ by $N+1$ random vectors $(Y_t)_{t \in \T}$ satisfying the consistency rule. The variable $Y$ is not explicit in $\mathcal{R}$ as it can be recovered from $T$ and $(Y_t)_{t \in \T}$ via consistency. In this setting, the first fundamental assumption for causal inference can be written as follows.
\begin{assumption}[Positivity]\label{hyp:positivity} $\O$ is such that for all $t \in \T$ and $\law{X}$-almost every $x$, $0 < \P(T=t \mid X=x) < 1$.
\end{assumption}
Remark that, strictly speaking, positivity in an hypothesis on the observational vector rather than the RCM. But people typically suppose it in the context of potential outcomes. We distinguish two formulations for the second fundamental assumption, namely conditional ignorability.
\begin{assumption}[Cross-outcome conditional ignorability]\label{hyp:cross_ignorability} $\mathcal{R}$ is such that $(Y_t)_{t \in \T} \independent T \mid X.$
\end{assumption}
\begin{assumption}[Single-outcome conditional ignorability]\label{hyp:single_ignorability} $\mathcal{R}$ is such that for all $t \in \T$, $Y_t \independent T \mid X$.
\end{assumption}
The stronger version (Assumption~\ref{hyp:cross_ignorability}) is the original one, but most causal-inference methods only require the weaker version (Assumption~\ref{hyp:single_ignorability}). In our main results, we will clearly specify which form of conditional ignorability is required. Crucially, Assumptions~\ref{hyp:positivity} and \ref{hyp:single_ignorability} permit to fully identify the law of $(T,X,Y_t)$ by the law of $(T,X,Y)$ for any $t \in \T$, as recalled below.
\begin{lemma}[Single-outcome identification of potential outcomes]\label{lm:sw}
Let $\O := (T,X,Y)$ be an observational vector satisfying positivity (Assumption~\ref{hyp:positivity}). For any $\mathcal{R} := (T,X,(Y_t)_{t \in \T}) \in \mathfrak{R}_{\O}$, if $\mathcal{R}$ meets single-outcome conditional ignorability (Assumption~\ref{hyp:single_ignorability}) then for every $t \in \T$ and $\law{X}$-almost any $x \in \R^d$:
\[
\mathcal{L}(Y_t \mid X=x) = \mathcal{L}(Y \mid X=x, T=t).
\]
Thereby, for any Borel set $F \subseteq \T \times \R^d \times \R^p$,
\[
\P((T,X,Y_t) \in F) = \int \P(Y \in F(t',x) \mid X=x, T=t) \mathrm{d}\P(X=x,T=t'),
\]
where $F(t',x) := \{y \in \R^p \mid (t',x,y) \in F\}$ for every $(t',x) \in \T \times \R^d$.
\end{lemma}

\subsubsection{Structural causal model}

On the other hand, we assume that the variables in $\O$ are generated by a latent SCM $\M$, as specified below.

\begin{definition}[Compatible structural causal model]\label{def:mo}
Let $\O$ be an observational vector. An SCM $\M := \langle U, g \rangle$ is \emph{compatible with $\O$} if it admits a unique solution $V$ equal to $\O$. We denote by $\mathfrak{M}_\O$ the class of SCMs satisfying acyclicity (Assumption~\ref{hyp:acyclic}) and compatible with $\O$.     
\end{definition}

The latent model $\M$ defines for every $t \in \T$ the post-intervention outcome $Y_{T=t}$ under $\operatorname{do}(T=t)$ using Definition~\ref{def:dointervention}. We refer to $Y_{T=t}$ as the \emph{structural counterfactual outcome} had the treatment been equal to $t$. Similarly to potential outcomes, structural counterfactuals satisfy the consistency rule, and therefore induce an RCM.
\begin{lemma}[Consistency rule for structural counterfactuals]\label{lem:consistency_do} Let $\O := (T,X,Y)$ be an observational vector. For any $\M \in \mathfrak{M}_\O$, the tuple $(Y_{T=t})_{t \in \T}$ obtained via do-interventions in $\M$ meets the consistency rule,
\[
Y \aseq \sum_{t \in \T} \mathbf{1}_{\{T=t\}} Y_{T=t}.
\]
Therefore, $\mathcal{R}_\M := (T,X,(Y_{T=t})_{t \in \T})$ belongs to $\mathfrak{R}_\O$. We call $\mathcal{R}_\M$ the \emph{entailed RCM} of $\M$.
\end{lemma}
The notion of entailed RCM bridges SCMs to RCMs, thereby enables one to compare these two classes of models which differ fundamentally in their constructions of counterfactual outcomes (see Remark~\ref{rem:versus}). Contrasting $\mathcal{R}_\M := (T,X,(Y_{T=t})_{t \in \T})$ against $\mathcal{R} := (T,X,(Y_t)_{t \in \T})$ almost surely or in law is precisely the goal of this paper, as formalized next in Section~\ref{sec:equivalence}.

\begin{remark}[Primitives versus derivatives]\label{rem:versus}
As noted by Pearl \cite{pearl2010brief}, the potential outcomes $(Y_t)_{t \in \T}$ are \say{undefined \emph{primitives}} of the RCM, not related to any formal of measurable quantities, while the post-intervention outcomes $(Y_{T=t})_{t \in \T}$ are \say{\emph{derivatives}} of the SCM by application of do-interventions. Said differently, the firsts are inputs \emph{defining} the causal model, whereas the seconds are post-intervention outputs \emph{defined by} the causal model. Figure~\ref{fig:principle} illustrates this aspect. As such, consistency is a \emph{theorem} for structural counterfactuals and a constitutive \emph{assumption} for potential outcomes. This holds more generally for any conceivable properties on counterfactual outcomes: in $\mathcal{R}$ they usually are assumptions; in $\mathcal{R}_{\M}$ they are derived from $\M$.
\end{remark}

\begin{figure}[t]
\centering

\begin{tikzpicture}[
  node distance=1.5cm and 2cm, 
  >={stealth},
  Box/.style={draw, minimum size=1cm, align=center}
]

\node[Box] (U) at (0,0) {$U$};
\node[Box] (A) [above left=of U] {$(Y_{\color{blue}t})_{t \in \T}$};
\node[Box] (B) [above=of U] {$T,X,Y$};
\node[Box] (C) [above right=of U] {$(Y_{\color{red}T=t})_{t \in \T}$};

\draw[->] (U) -- (B) node[midway, left] {$g$};
\draw[->] (U) -| (C);

\draw[->, red, dashed] ($(U)!.5!(B)+(0.2,-0.2)$) -- ($(U)!.5!(C)+(+1.6,-0.2)$)
    node[midway, above] {$\operatorname{do}(T=t)$\ $ \forall t \in \T$};

\draw[decorate, decoration={brace, amplitude=10pt, raise = 20pt}, blue] 
    ($(A)+(-0.2,0)$) -- ($(B)+(-0.3,0)$)
    node[midway, yshift=30pt, above] {$\mathcal{R}$};

\draw[decorate, decoration={brace, amplitude=10pt, raise = 20pt}, red] 
    ($(B)+(0.3,0)$) -- ($(C)+(0.2,0)$)
    node[midway,yshift=30pt, above] {$\mathcal{R}_{\M}$};

\draw[decorate, decoration={brace, amplitude=10pt, raise =20pt}, red] 
    (U)+(0,-0.3) -- ($(B)+(0,-0.6)$)
    node[midway,xshift=-30pt, left] {$\M$};

\node at (B) [above=0.7cm] {$\O$};

\node at (A) [below=0.7cm, text width=2.5cm] {properties are \emph{assumptions}};

\node at (C) [right=1cm, text width=2.5cm] {properties are \emph{derived} from $\M$};

\draw[->, blue, dashed] ($(A)+(-1.5,0)$) -- (A); 

\end{tikzpicture}

\caption{Superimposed construction of an RCM $\mathcal{R}$ and an SCM $\M$ compatible with a same observational vector $\O$.}
\label{fig:principle}
\end{figure}

\subsection{Notions of equivalence between causal models}\label{sec:equivalence}

We aim at studying the mathematical similarities and differences between potential outcomes and structural counterfactuals from a theoretically neutral perspective. We consider three levels of comparison that will guide our analysis throughout the paper.

\begin{definition}[Equivalences between causal models]\label{def:representation}
Let $\O$ be an observational vector, and $\mathcal{R}_1 := (T,X,(Y^{(1)}_t)_{t \in \T}),\mathcal{R}_2 := (T,X,(Y^{(2)}_t)_{t \in \T})$ be two models in $\mathfrak{R}_\O$. We say that $\mathcal{R}_1$ and $\mathcal{R}_2$ are:
\begin{itemize}
    \item[(i)] \emph{almost-surely equivalent} if $Y^{(1)}_t \aseq Y^{(2)}_t$ for any $t \in \T$;
    \item[(ii)] \emph{cross-outcome equivalent} if $\law{(T,X,(Y^{(1)}_t)_{t \in \T})} = \law{(T,X,(Y^{(2)}_t)_{t \in \T})}$;
    \item[(iii)] \emph{single-outcome equivalent} if $\law{(T,X,Y^{(1)}_t)} = \law{(T,X,Y^{(2)}_t)}$ for every $t \in \T$.
\end{itemize}
By overloading terminology, for $\M \in \mathfrak{M}_\O$ and $\mathcal{R} \in \mathfrak{R}_\O$, we analogously say that $\M$ is equivalent to $\mathcal{R}$ in one of the above senses if $\mathcal{R}_\M$ is equivalent to $\mathcal{R}$ in this sense.
\end{definition}
These types of equivalence focus on (entailed) RCMs in contrast to the notions proposed in \citep{bongers2021foundations, beckers2021equivalent} which are specifically tailored to SCMs. Remark that (i) $\implies$ (ii) $\implies$ (iii). Mathematically, each relation characterizes a class of RCM-based estimands that are invariant under the swap of two equivalent models $\mathcal{R}_1$ and $\mathcal{R}_2$, which amounts to exchanging their potential outcomes. If the models are almost-surely equivalent, then \emph{all} estimands are invariant under the swap. If the models are cross-outcome equivalent, then any estimands \emph{based on their cross-outcome distributions} are invariant under the swap: for instance, $\E \left[\lVert Y^{(1)}_1 - Y^{(1)}_0 \rVert^2 \mid X=x \right] = \E \left[ \lVert Y^{(2)}_1 - Y^{(2)}_0 \rVert^2 \mid X=x \right]$. If the models are single-outcome equivalent, then any estimands \emph{based on their single-outcome distributions} are invariant under the swap: for example, $\E \left[ Y^{(1)}_1 - Y^{(1)}_0 \mid X=x \right] = \E \left[Y^{(2)}_1 - Y^{(2)}_0 \mid X=x \right]$. When equivalence does not hold, some causal effects of the corresponding relation may vary under the swap. Less formally, the motivation for the three ladders in Definition~\ref{def:representation} comes from different levels at which people reason counterfactually. Let us detail them with an RCM and an SCM.

The almost-sure level focuses on counterfactual questions at the scale of $\omega \in \Omega$. It asks (for example) \say{What would have been the value of $Y(\omega)$ had $T(\omega)$ been equal to 1 instead of 0?}. The answer is deterministic in both the potential-outcome approach and the structural approach, given by $Y_1(\omega)$ in the RCM and by $Y_{T=1}(\omega)$ in the SCM. Should the models be almost-surely equivalent, the answer would be identical for almost-every $\omega \in \Omega$. We emphasize that this level is methodologically inessential: in practice, one does not have access to the random variables of the models themselves but to some realizations of their laws. Nevertheless, comparing causal models on this almost-sure baseline is theoretically important to completely understand their mathematical differences.

We now turn to the cross-outcome level. Because counterfactual questions at the scale of $\omega$ cannot be answered, people rather address surrogate queries like \say{What would have been the law of $Y$ had $T$ been equal to 1 instead of 0 given that $X=x$?}. In both causal approaches, the answer is generally not deterministic (that is uniquely determined). Possible answers along with their probabilities are respectively described by $\law{Y_1 \mid X=x}$ (which can be estimated using the techniques mentioned in Section~\ref{sec:causal_estimation}) in the potential-outcome framework and by $\law{Y_{T=1} \mid X=x}$ (which can be inferred via the three-step procedure from Section~\ref{sec:3steps}) in the SCM. More generally, most answers to counterfactual questions in the structural framework and in the potential-outcome framework are respectively characterized by the joint probability distributions $\mathcal{L}((T,X,(Y_{T=t})_{t \in \T}))$ and $\mathcal{L}((T,X,(Y_t)_{t \in \T}))$. Therefore, should the models be cross-outcome equivalent, they would yield the same conclusions when reasoning counterfactually at this level.

The single-outcome level resembles the cross-outcome level in the sense that it also focuses on distributions of outcomes; it differs by being mathematically weaker. It is motivated by the fact that researchers and practitioners predominantly ask counterfactual questions involving single counterfactual outcomes instead of joint counterfactual outcomes, as in the above paragraph. As such, knowing only the marginal laws $\mathcal{L}((T,X,Y_t))$ and $\mathcal{L}((T,X,Y_{T=t}))$ for every $t \in \T$ suffices to compute most of the practically relevant counterfactual estimands. This is why it is arguably the most critical level in practice. For illustration, consider average treatment effects like $\E[Y_1 - Y_0]$ or $\E[Y_1 - Y_0 \mid X=x]$ (as in Section~\ref{sec:ex_intro}), distributional treatment effects like $D(\law{Y_1},\law{Y_0})$ or $D(\law{Y_1 \mid X=x},\law{Y_0 \mid X=x})$ where $D$ is a discrepancy between probability measures \citep{muandet2021counterfactual,park2021conditional}, or the \emph{standard} counterfactual-fairness condition on $Y$ : $\law{Y_{T=t'} \mid X=x, T=t} = \law{Y_{T=t} \mid X=x, T=t}$ for every $t \in \T$ \citep{kusner2017counterfactual}. They all concern the single-outcome level.\footnote{There also exists a notion of \emph{strong} counterfactual fairness \citep[Equation 4]{kusner2017counterfactual}. It demands $\P(Y_{T=t} = Y_{T=t'} \mid X=x, T=t) = 1$ for every $t,t' \in \T$, and is thereby a proper \emph{cross-outcome} condition.} Therefore, should the models be single-outcome equivalent, they would yield the same results in these state-of-the-art methodologies.

\begin{remark}[At which level do the fundamental assumptions matter?]\label{rem:level}
According to Lemma~\ref{lm:sw}, positivity (Assumption~\ref{hyp:positivity}) and single-outcome conditional ignorability (Assumption~\ref{hyp:single_ignorability}) uniquely determines the single-outcome level of RCMs in $\mathfrak{R}_{\O}$ by $\law{\O}$. As such, these assumptions allow the estimation of most practically-relevant causal effects. However, we point out that even cross-outcome conditional ignorability (Assumption~\ref{hyp:cross_ignorability}) does not uniquely determines the cross-outcome level. In Remark~\ref{rem:identification} from Section~\ref{sec:comparison}, we study two RCMs $\mathcal{R}_1 := (T,X,(Y^{(1)}_t)_{t \in \T})$ and $\mathcal{R}_2 := (T,X,(Y^{(2)}_t)_{t \in \T})$ satisfying positivity and cross-outcome conditional ignorability such that $\law{(T,X,(Y^{(1)}_t)_{t \in \T})} \neq \law{(T,X,(Y^{(2)}_t)_{t \in \T})}$. All in all, the scope of the fundamental assumptions of causal inference is limited to the single-outcome level of counterfactual reasoning. Identifying a cross-outcome effect such as $\E[Y_1 \mid Y=0, T=0]$ requires stronger assumptions.
    
\end{remark}

To summarize, when two causal models compatible with a same observational vector meet an equivalence condition from Definition~\ref{def:representation}, they can be used interchangeably at a certain level of counterfactual reasoning. This mathematically signifies that a certain class of causal effects is invariant under exchanges of their counterfactual outcomes (\textit{e.g.}, the CATE for single-outcome equivalent models). In the next section, we study whether $\M \in \mathfrak{M}_\O$ and $\mathcal{R} \in \mathfrak{R}_\O$ need to, do not need to, can, or cannot be equivalent under different degrees of assumptions.

\section{Main results}\label{sec:main}

In this section, we compare the generic RCM and SCM introduced in Section~\ref{sec:models} according to the three levels presented in Section~\ref{sec:equivalence}. More precisely, Section~\ref{sec:negative} firstly reminds that equivalence does not necessarily hold whatever the level, then Section~\ref{sec:comparison} identifies and studies cases where equivalence does (not) hold at the single-outcome level. Finally, Section~\ref{sec:interpretation} proposes conceptual and practical interpretations of these mathematical results.

\subsection{Equivalence does not generally hold}\label{sec:negative}

We start by a crucial reminder justifying why it is relevant to compare the models at the aforementioned levels: an RCM and an SCM compatible with a same observational vector are not necessarily equivalent.

\subsubsection{General case}

The proposition below formalizes this claim in the most general scenario.

\begin{proposition}[Equivalence is not necessary]\label{prop:nas}
Let $\O$ be an observational vector. For any $\M \in \mathfrak{M}_\O$, there exists $\mathcal{R} \in \mathfrak{R}_\O$ such that $\M$ and $\mathcal{R}$ are equivalent in none of the senses from Definition~\ref{def:representation}.
\end{proposition}

This result rests on a simple remark: the consistency rule does not suffice to characterize potential outcomes $\P$-almost surely.\footnote{This formally means that the set $\mathfrak{R}_\O$ is not reduced to one class of almost-surely equivalent RCMs.} More precisely, while necessarily $Y_t(\omega) = Y(\omega) = Y_{T=t}(\omega)$ on $\{T=t\}$ for $t \in \T$ (up to $\P$-negligible subsets) due to Lemma~\ref{lem:consistency_do}, there is no constraint on $Y_t(\omega)$ over $\Omega \setminus \{T=t\}$; it could take any value over it without violating the consistency rule. In contrast, $Y_{T=t}$ is defined (almost) everywhere through the altered SCM $\mathcal{M}_{T=t}$. The proof of Proposition~\ref{prop:nas} exploits this specification issue of potential outcomes: for any $t \in \T$, $Y_t$ can be any function on $\{T \neq t\}$, thereby can be chosen distinct to $Y_{T=t}$.
\begin{proof}\textbf{of Proposition~\ref{prop:nas}}\mbox{ }
Let $(Y_{T=t})_{t \in \T}$ be the structural counterfactuals of $\M$, and define potential outcomes $(Y_t)_{t \in \T}$ as follows. For any $t \in \T$,
\[
    Y_t := \mathbf{1}_{\{T=t\}} Y_{T=t} + \mathbf{1}_{\{T \neq t\}} (Y_{T=t} + y),
\]
where $y \in \R^p$ is not the null vector. The tuple $(Y_t)_{t \in \T}$ satisfies the consistency rule according to Lemma~\ref{lem:consistency_do}. Moreover, for any $t \in \T$, $Y_t$ is not almost-surely equal nor equal in law to $Y_{T=t}$ due to $\P(T \neq t)>0$ and $Y_{T=t}+y \neq Y_{T=t}$.
\end{proof}
It may seem counter-intuitive to design oneself the potential outcomes (as done in the proof and in upcoming examples), since people frequently consider them as externally imposed. We emphasize that Proposition~\ref{prop:nas} is theoretically neutral, and only serves to remind that equality between counterfactual outcomes across causal models does not \emph{necessarily} hold.

\subsubsection{Causal-inference setting}

Proposition~\ref{prop:nas} focuses on the most general setting where the potential outcomes meet only consistency. Said differently, it simply shows that consistency alone is not enough to guarantee equality (almost-surely or in law) between counterfactual outcomes. Nevertheless, people suppose most often that the potential outcomes also satisfy conditional ignorability in order to apply causal-inference techniques. This raises the question whether such an additional hypothesis could render the causal models equivalent. The next theorem ensures that equivalence does not always hold even under the fundamental assumptions of causal inference.
\begin{proposition}[Equivalence is not necessary under the fundamental assumptions of causal inference]\label{prop:nas_fa} Let $\O$ be an observational vector satisfying positivity (Assumption~\ref{hyp:positivity}).
There exist $\M \in \mathfrak{M}_\O$ and $\mathcal{R} \in \mathfrak{R}_\O$ satisfying conditional ignorability (Assumption~\ref{hyp:cross_ignorability} or \ref{hyp:single_ignorability}) such that $\M$ and $\mathcal{R}$ are equivalent in none of the senses from Definition~\ref{def:representation}.
\end{proposition}

This results comes from the fact that conditional ignorability fully determines the single-outcome distributions of an RCM (Lemma~\ref{lm:sw}) and does not hold by design in an entailed RCM.

\begin{proof}\textbf{of Proposition~\ref{prop:nas_fa}}\mbox{ }
We address the specific case where $\T := \{0,1\}$ and both $X$ and $Y$ are $\R$-valued; one can readily generalize the proof. Consider the following SCM:
\begin{align*}
    T &\aseq U_T,\\
    X &\aseq T + U_X\\
    Y &\aseq T + X + U_Y,
\end{align*}
where $U_T$ follows a Bernoulli distribution with parameter $1/2$, while $U_X$ and $U_Y$ both follow the centered Gaussian distribution with unit variance, such that $U_T, U_X$ and $U_Y$ are mutually independent.

By do-interventions, $X_{T=t} \aseq t + U_X$ and $Y_{T=t} \aseq t + X_{T=t} + U_Y$ for any $t \in \T$. Therefore, by substituting the expression of $X_{T=t}$ in the expression of $Y_{T=t}$ we obtain $Y_{T=0} \aseq U_X + U_Y$ and $Y_{T=1} \aseq 2+ U_X + U_Y$. Then, we define potential outcomes as:
\begin{align}
Y_0 &:= (1-T) \cdot (X + U_Y) + T \cdot (X - U_Y),\label{eq:y0_def}\\
Y_1 &:= (1-T) \cdot (1 + X - U_Y) + T \cdot (1 + X + U_Y).\label{eq:y1_def}
\end{align}
A key property of this construction is that $\law{U_Y} = \law{-U_Y}$ by symmetry of $U_Y$'s distribution, while $\P(U_Y = -U_Y) = 0$. After simplification,
\begin{align}
Y_0 &= (1-T) \cdot (T + U_X + U_Y) + T \cdot (T + U_X - U_Y) = T + (1-2T) \cdot U_Y + U_X,\label{eq:y0_simple}\\
Y_1 &= (1-T) \cdot (1 + T + U_X - U_Y) + T \cdot (1 + T + U_X + U_Y) = (1 + T) - (1-2T) \cdot U_Y + U_X.\label{eq:y1_simple}
\end{align}

Let us check the required assumptions using Equations~\eqref{eq:y0_def} and \eqref{eq:y1_def}. Firstly, the pair $(Y_0,Y_1)$ clearly satisfies $Y \aseq (1-T) \cdot Y_0 + T \cdot Y_1$. Secondly, for $\law{X}$-almost every $x \in \R$ and any $t \in \T$, $\P(T=t \mid X=x) = \frac{\varphi(x-t)}{\varphi(x) + \varphi(x-1)} > 0$ where $\varphi$ is the density function of the centered Gaussian distribution with unit variance. Therefore, positivity holds. Thirdly, $\law{(Y_0,Y_1) \mid X=x, T=0} = \law{(x+U_Y,1+x-U_Y) \mid X=x, T=0}$ by definition. Next, Lemma~\ref{lm:noise} ensures that $U_Y \independent (T,X)$, thereby $\law{(Y_0,Y_1) \mid X=x, T=0} = \law{(x+U_Y,1+x-U_Y)} = \law{(x-U_Y,1+x+U_Y)}$ since $\law{U_Y} = \law{-U_Y}$. Then, notice that
\[
\law{(Y_0,Y_1) \mid X=x, T=1)} = \law{(x-U_Y, 1+x+U_Y}) = \law{(Y_0,Y_1) \mid X=x, T=0}
\]
using once again $U_Y \independent (T,X)$. Therefore, cross-outcome conditional ignorability holds. To conclude, verify from Equations~\eqref{eq:y0_simple} and \eqref{eq:y1_simple} that $\E[Y_0] = 1/2$ while $\E[Y_{T=0}] = 0$, and $\E[Y_1] = 3/2$ while $\E[Y_{T=1}] =2$. Therefore, $\law{Y_t} \neq \law{Y_{T=t}}$ for any $t \in \T$.

This proves the result with an $\mathcal{R}$ that satisfies cross-outcome conditional ignorability (Assumption~\ref{hyp:cross_ignorability}). We now show that the same holds when only single-outcome conditional ignorability (Assumption~\ref{hyp:single_ignorability}) is true. To do so, we slightly modify the definition of $Y_0$ in the above example while keeping $Y_1$ identical. Set $Y_0 := X + U_Y = T + U_X + U_Y$, so that $\E[Y_0] = 1/2$ again. Hence, we still have that $\law{Y_t} \neq \law{Y_{T=t}}$ for any $t \in \T$. Moreover, note that the consistency rule still holds. It remains to verify conditional ignorability.

Let us check that single-outcome conditional ignorability \emph{does} hold. For $t \in \T$,
\begin{align*}
    \law{Y_0 \mid X=x, T=t} &= \law{x + U_Y}\\
    \law{Y_1 \mid X=x, T=t} &= (1-t) \cdot \law{1+x-U_Y} + t \cdot \law{1 + x + U_Y} = \law{1+x+U_Y},
\end{align*}
due to $U_Y \independent(T,X)$ and $\law{U_Y} = \law{-U_Y}$. From the above expressions and the fact that $U_Y \independent T$, it follows that $Y_t \independent T \mid X$ for $t \in \T$. We turn to proving that cross-outcome conditional ignorability \emph{does not} hold. By a similar calculus to before,
\begin{align*}
    \law{(Y_0,Y_1) \mid X=x, T=0} &= \law{(x + U_Y, 1 + x - U_Y)}\\
    \law{(Y_0,Y_1) \mid X=x, T=1} &= \law{(x + U_Y, 1 + x + U_Y)}.
\end{align*}
Therefore, $\law{(Y_0,Y_1) \mid X=x, T=0} \neq \law{(Y_0,Y_1) \mid X=x, T=1}$. This concludes the proof.
\end{proof}

To summarize Section~\ref{sec:negative}: whether it be at the variable or distributional level, with or without the fundamental assumptions of causal inference, potential outcomes and structural counterfactuals are not necessarily equal. Therefore, whatever the level of counterfactual reasoning, using $(T,X,(Y_t)_{t \in \T})$ and $(T,X,(Y_{T=t})_{t \in \T})$ interchangeably (as commonly done in the scientific literature) must be either the manifestation of an arbitrary choice---a selection among all the pairs of equivalent causal models---or the mathematical consequence of further assumptions. Notably, Proposition~\ref{prop:nas_fa} demands strong hypotheses on the potential outcomes but let the latent SCM almost unrestrained. This motivates a sharper analysis of the conditions (in particular on the SCM) that could imply equivalence or on the contrary make it impossible. This is precisely what the following subsection proposes. 

\subsection{Comparison of causal models in the causal-inference setting}\label{sec:comparison}

In what follows, we aim at providing from a purely mathematical perspective a better understanding of what renders (or not) an SCM and an RCM equivalent.

\subsubsection{Structural assumptions}

We start by introducing additional assumptions on the underlying SCM that will simplify the comparison with an RCM. Let us denote by $U_T := U_{\operatorname{Exo}(T)}$, $U_X  := U_{\operatorname{Exo}(X)}$ and $U_Y := U_{\operatorname{Exo}(Y)}$ the exogenous parents of respectively $T$, $X$, and $Y$ in the latent SCM $\M := \langle g, U \rangle$. These notations have already been used in the SCMs from the introductory example of Section~\ref{sec:ex_intro} and from the proof of Proposition~\ref{prop:nas_fa} (in the specific case where $T$, $X$, and $Y$ were all univariate). Then, we consider the following graphical conditions on $\M$:
\begin{assumption}[Outcome]\label{hyp:outcome} $\M$ is such that $Y_{\operatorname{Endo}(T)} = Y_{\operatorname{Endo}(X)} = \emptyset$.
\end{assumption}
\begin{assumption}[Independent random noises]\label{hyp:noises} $\M$ is such that $\operatorname{Exo}(Y) \cap \left(\operatorname{Exo}(T) \cup \operatorname{Exo}(X)\right) = \emptyset$, which means that $U_Y \independent (U_T,U_X)$.
\end{assumption}
Assumption~\ref{hyp:outcome} structurally defines the variable $Y$ as the \emph{outcome}; it changes in response to $X$ and $T$ but not the contrary. Through Lemma~\ref{lem:docalculus}, it permits to write
\[
\begin{cases}
    T &\aseq f_T(X,U_T),\\
	X &\aseq f_X(T,U_X),\\
    Y &\aseq f_Y(T,X,U_Y),
\end{cases}
\hspace{1cm}\text{and}\hspace{1cm}
\begin{cases}
    T_{T=t} &\aseq t,\\
	X_{T=t} &\aseq f_X(t,U_X),\\
    Y_{T=t} &\aseq f_Y(t,X_{T=t},U_Y),
\end{cases}  
\]
for any $t \in \T$, where $f_T, f_X$ and $f_Y$ are measurable functions derived from $g$. The artificial cycle in these formulas (\emph{i.e.}, $X$ and $T$ are both functions of each other) merely serves to consider all configurations of causal links between $T$ and $X$ (see Figure~\ref{fig:decisive}); strictly, $\mathcal{M}$ satisfies Assumption~\ref{hyp:acyclic}.

Assumption~\ref{hyp:noises} is a random-noise condition that notably holds in the widely-used class of SCMs referred to as Markovian, where all the noises in $(U_{\operatorname{Exo}(i)})_{i \in \I}$ are mutually independent. As clarified by the following lemma, this item guarantees that all potential confounders between $T$ and $Y$---except $T$ itself---are included in $X$.
\begin{lemma}[No hidden confounder]\label{lm:noise}
Let $\O := (T,X,Y)$ be an observational vector. For any $\M \in \mathfrak{M}_\O$ satisfying Assumptions~\ref{hyp:outcome} and \ref{hyp:noises}, if $U_T,U_X$ and $U_Y$ denote the exogenous parents of respectively $T,X$ and $Y$ in $\M$, then $U_Y \independent (T,X)$.   
\end{lemma}
Altogether, Assumptions~\ref{hyp:outcome} and \ref{hyp:noises} mean that the randomness of the outcome $Y \aseq f_Y(T,X,U_Y)$ can be causally interpreted by three sources: the direct effect of the treatment status $T$, the direct effect of the covariates $X$, and any other possible effects $U_Y$ independent to $T$ and $X$.

\subsubsection{Single-outcome comparison}\label{sec:theorem}

Supposing that the underlying SCM satisfies the above assumptions, we derive theoretical results on the single-outcome level of equivalence between causal models. To contrast at this level an RCM with the SCM, we need a common basis to compare laws across them. The theorem below characterizes the law of $(T,X,Y_t)$ for any $t \in \T$ \emph{through the latent SCM $\mathcal{M}$} under the fundamental assumptions of causal inference, thereby enabling us to compare $\law{(T,X,Y_t)}$ with $\law{(T,X,Y_{T=t})}$.
\begin{theorem}[Single-outcome structural expression of potential outcomes]\label{thm:identification}
Let $\O := (T,X,Y)$ be an observational vector satisfying positivity (Assumption~\ref{hyp:positivity}), and $\mathcal{M} \in \mathfrak{M}_\O$ meet Assumptions~\ref{hyp:outcome} plus \ref{hyp:noises}. According to Lemma~\ref{lm:noise}, $U_Y \independent (T,X)$ where $U_Y$ denotes the exogenous parents of $Y$ in $\M$. According to Lemma~\ref{lem:docalculus}, there exists a measurable function $f_Y$ such that $Y \aseq f_Y(T,X,U_Y)$ and $Y_{T=t} \aseq f_Y(t, X_{T=t},U_Y)$ for any $t \in \T$. Then, for any $\mathcal{R} := (T,X,(Y_t)_{t \in \T}) \in \mathfrak{R}_\O$ satisfying conditional ignorability (Assumption~\ref{hyp:single_ignorability}),
\[
    \law{(T,X,Y_t)} = \law{(T,X,f_Y(t,X,U_Y))},
\]
for any $t \in \T$.
\end{theorem}
The proof combines the RCM single-outcome identification from Lemma~\ref{lm:sw} with the SCM expression of $Y$ from Lemma~\ref{lem:docalculus} to connect the two causal models at the single-outcome level.
\begin{proof}\textbf{of Theorem \ref{thm:identification}}\mbox{ }Let $t \in \T$. Let $F \subseteq \T \times \R^d \times \R^p$ be a Borel set, and define the sets $F(t',x) := \{y \in \R^p \mid (t',x,y) \in F\}$ for every $(t',x) \in \T \times \R^d$. According to Lemma~\ref{lm:sw},
\begin{equation}\label{eq:sw_identification}
    \P((T,X,Y_t) \in F) = \int \P(Y \in F(t',x) \mid X=x,T=t) \mathrm{d}\P(X=x,T=t').
\end{equation}
Additionally, according to Lemma~\ref{lem:docalculus} the pre-intervention and post-intervention outcomes can be written as $Y \aseq f_Y(T,X,U_Y)$ and $Y_{T=t} \aseq f_Y(t,X_{T=t},U_Y)$. Therefore, for any $t' \in \T$ and $\law{X}$-almost any $x \in \R^d$,
\[
    \mathcal{L}(Y \mid X=x, T=t) = \mathcal{L}(f_Y(T,X,U_Y) \mid X=x, T=t) = \mathcal{L}(f_Y(t,x,U_Y) \mid X=x, T=t).
\]
Also, it follows from Assumptions~\ref{hyp:outcome} and \ref{hyp:noises} through Lemma~\ref{lm:noise} that $U_Y \independent (T,X)$, hence that $\mathcal{L}(f_Y(t,x,U_Y) \mid X=x, T=t) = \mathcal{L}(f_Y(t,x,U_Y) \mid X=x, T=t')$. By continuing the above computation we therefore have for any $t' \in \T$:
\begin{equation}\label{eq:scm_identification}
    \mathcal{L}(Y \mid X=x, T=t) = \mathcal{L}(f_Y(t,x,U_Y) \mid X=x, T=t').
\end{equation}
Combining \eqref{eq:sw_identification} with \eqref{eq:scm_identification} leads to
\[
    \P((T,X,Y_t) \in F) = \int \P( f_Y(t,x,U_Y) \in F(t',x) \mid X=x,T=t')\mathrm{d}\P(X=x,T=t').
\]
Finally,
\begin{align*}
\P((T,X,Y_t) \in F) &= \int \P( (T,X,f_Y(t,X,U_Y)) \in F \mid X=x,T=t')\mathrm{d}\P(X=x,T=t')\\ &= \P( (T,X,f_Y(t,X,U_Y)) \in F).
\end{align*}
This means that $\mathcal{L}((T,X,Y_t)) = \mathcal{L}((T,X,f_Y(t,X,U_Y)))$, which concludes the proof.
\end{proof}

\begin{remark}[Single-outcome expression vs.~cross-outcome expression]\label{rem:identification}
Theorem~\ref{thm:identification} expresses for any $t \in \T$ the marginal distribution $\law{(T,X,Y_t)}$ with SCM-based quantities under single-outcome conditional ignorability (Assumption~\ref{hyp:single_ignorability}), but does not provide a similar formula for the whole joint distribution $\law{(T,X,(Y_t)_{t \in \T})}$. There is an explanation: $\law{(T,X,(Y_t)_{t \in \T})} \neq \law{(T,X,(f_Y(t,X,U_Y))_{t \in \T})}$ in general---even under cross-outcome conditional ignorability (Assumption~\ref{hyp:cross_ignorability}). To prove this point, we consider the same setting as the proof of Proposition~\ref{prop:nas_fa}, in which Assumptions~\ref{hyp:positivity}, \ref{hyp:cross_ignorability}, \ref{hyp:outcome} and \ref{hyp:noises} hold.

Using the SCM notations, $f_Y(0,X,U_Y) = X + U_Y$ and $f_Y(1,X,U_Y) = 1+ X + U_Y$. It then follows from $\law{-U_Y} = \law{U_Y}$ and Equations~\eqref{eq:y0_def} and \eqref{eq:y1_def} that $\law{Y_0}=\law{f_Y(0,X,U_Y)}$ and $\law{Y_1}=\law{f_Y(1,X,U_Y)}$, as anticipated by Theorem~\ref{thm:identification}. However, $\law{(Y_0,Y_1)} \neq \law{(f_Y(0,X,U_Y),f_Y(1,X,U_Y))}$. It suffices to remark that $\law{f_Y(1,X,U_Y)-f_Y(0,X,U_Y)} = \delta_1$ (the Dirac measure at $1$) whereas $\law{Y_1-Y_0} = \law{1 + 2(2T-1) \cdot U_Y} \neq \delta_1$ using Equations~\eqref{eq:y0_simple} and \eqref{eq:y1_simple}.

This remark is consistent with the point made in Remark~\ref{rem:level} that we shall prove now. Continuing the same example, we define the potential outcomes $(Y^{(1)}_0,Y^{(1)}_1) := (Y_0,Y_1)$ and $(Y^{(2)}_0,Y^{(2)}_1) := (f_Y(0,X,U_Y),f_Y(1,X,U_Y))$ along with their associated RCMs $\mathcal{R}^{(1)}$ and $\mathcal{R}^{(2)}$ in $\mathfrak{R}_{\O}$. Note that the pair $(Y^{(2)}_0,Y^{(2)}_1)$ clearly satisfies the consistency rule. Moreover,
\begin{align*}
\law{(Y^{(2)}_0,Y^{(2)}_1) \mid X=x, T=0} &= \law{(x+U_Y,1+x+U_Y) \mid X=x,T=0}\\ &= \law{(x+U_Y,1+x+U_Y) \mid X=x,T=1}\\ &= \law{(Y^{(2)}_0,Y^{(2)}_1) \mid X=x, T=1},
\end{align*}
due to $U_Y \independent (T,X)$. Therefore, $\mathcal{R}^{(1)}$ and $\mathcal{R}^{(2)}$ both meet positivity and cross-outcome conditional ignorability, but are not cross-outcome equivalent.
\end{remark}

As such, Theorem~\ref{thm:identification} enables one to compare potential outcomes and structural counterfactuals at the single-outcome level. It critically implies under the fundamental assumptions of causal inference that for any $t \in \T$:
\begin{equation}\label{eq:compare}
    \mathcal{L}(Y_t) = \mathcal{L}(f_Y(t,X,U_Y))\ \text{and}\ \mathcal{L}(Y_{T=t}) = \mathcal{L}(f_Y(t,X_{T=t},U_Y)).
\end{equation}
This result sheds light on Proposition~\ref{prop:nas_fa}: $Y_t$ and $Y_{T=t}$ are not necessarily equal in law since $\mathcal{L}(X) \neq \mathcal{L}(X_{T=t})$ in general. Moreover, it furnishes sufficient conditions on the latent SCM for single-outcome equivalence to hold. Observe that the probability distributions from Equation~\eqref{eq:compare} are equal if: (1) $X$ is not altered by do-interventions on $T$, or (2) $Y$ is not impacted by $X$. The following assumption captures scenario (1), which is the most relevant in causal inference.
\begin{assumption}[No posttreatment covariate]\label{hyp:control} $\M$ is such that, $T_{\operatorname{Endo}(X)} = \emptyset$. 
\end{assumption}
Assumption~\ref{hyp:control} together with Assumption~\ref{hyp:outcome} mean that $X$ contains no descendant of $T$ in $\M$. It implies that $X_{T=t} \aseq X$ for any $t \in \T$, which guarantees single-outcome equivalence under the assumptions of Theorem~\ref{thm:identification}.
\begin{corollary}[Single-outcome equivalence under no posttreatment covariate]\label{cor:equivalence}
Let $\O$ be an observational vector satisfying positivity (Assumption~\ref{hyp:positivity}), $\mathcal{M} \in \mathfrak{M}_\O$ meet Assumptions~\ref{hyp:outcome} plus \ref{hyp:noises}, and $\mathcal{R} \in \mathfrak{R}_\O$ satisfy single-outcome conditional ignorability (Assumption~\ref{hyp:single_ignorability}). If Assumption~\ref{hyp:control} holds, then $\M$ and $\mathcal{R}$ are single-outcome equivalent but not necessarily cross-outcome equivalent, even under cross-outcome conditional ignorability (Assumption~\ref{hyp:cross_ignorability}).
\end{corollary}
\begin{proof}\textbf{of Corollary~\ref{cor:equivalence}}\mbox{ }
Let $t \in \T$. Under the considered assumptions, we have according to Theorem~\ref{thm:identification}$:\law{(T,X,Y_t)} = \law{(T,X,f_Y(t,X,U_Y))}$. Moreover, recall that $Y_{T=t} \aseq f_Y(t,X_{T=t},U_Y)$. If $X$ contains no descendant of $T$ in $\M$, then $X_{T=t} \aseq X$, leading to $\law{(T,X,Y_t)} = \law{(T,X,Y_{T=t})}$.

We now turn to proving that cross-outcome equivalence does not necessarily hold, even under cross-outcome conditional ignorability. To this end, we consider the following SCM:
\begin{align*}
    T &\aseq U_T,\\
    X &\aseq U_X,\\
    Y &\aseq T + X + U_Y,
\end{align*}
where $U_T$ follows a Bernoulli distribution with parameter $1/2$, $U_X$ is any $\R$-valued random variable, and $U_Y$ follows a centered Gaussian distribution with unit variance, such that $U_T,U_X$ and $U_Y$ are mutually independent. This model satisfies Assumptions~\ref{hyp:outcome}, \ref{hyp:noises} and \ref{hyp:control} by construction. By do-interventions, $Y_{T=0} \aseq U_X + U_Y$ and $Y_{T=1} = 1+U_X +U_Y$. Next, define the potential outcomes as follows:
\begin{align*}
    Y_0 &:= (1-T) \cdot(X + U_Y) + T \cdot (X - U_Y) = U_X +(1-2T) \cdot U_Y,\\
    Y_1 &:= (1-T) \cdot (1+X-U_Y) + T \cdot (1+X+U_Y) = 1 + U_X - (1-2T) \cdot U_Y.
\end{align*}
Let us verify the required assumptions. Firstly, the potential outcomes clearly satisfy consistency, that is $Y = (1-T) \cdot Y_0 + T \cdot Y_1$. Secondly, since $T \independent X$, we have $\P(T=1 \mid X) = \P(T=1) = 1/2$ which entails positivity. Thirdly, $\mathcal{L}((Y_0,Y_1) \mid X=x, T=1) = \mathcal{L}((Y_0,Y_1) \mid T=1)$ because $U_Y \independent (T,X) \aseq (U_T,U_X)$. Additionally, $\mathcal{L}((Y_0,Y_1) \mid T=1) = \mathcal{L}((U_X-U_Y,1+U_X+U_Y)) = \mathcal{L}((U_X+U_Y, 1+U_X-U_Y))$ since $\mathcal{L}(U_Y) = \mathcal{L}(-U_Y)$, and $\mathcal{L}((U_X+U_Y, 1+U_X-U_Y)) = \mathcal{L}((Y_0,Y_1) \mid T=0) = \mathcal{L}((Y_0,Y_1) \mid X=x, T=0)$ using again $U_Y \independent (T,X)$. Wrapping up: $\mathcal{L}((Y_0,Y_1) \mid X=x, T=1) = \mathcal{L}((Y_0,Y_1) \mid X=x, T=0)$, meaning that cross-outcome conditional ignorability holds according to the first part of the proof.

Therefore, the causal models are single-outcome equivalent. To conclude that they are not cross-outcome equivalent, note that $\law{Y_1-Y_0} = \law{1 + 2(2T-1) \cdot U_Y}$ whereas $\law{Y_{T=1}-Y_{T=0}} = \delta_1$.
\end{proof}

This result specifies a large class of SCMs that are single-outcome equivalent to an RCM meeting positivity and conditional ignorability: those in which $Y$ is a conditionally unconfounded outcome (Assumptions~\ref{hyp:outcome} plus \ref{hyp:noises}) and $T$ does not impact $X$ (Assumption~\ref{hyp:control}). If one of the these two sets of hypotheses does not hold, then single-outcome equivalence is not guaranteed. We emphasize that Assumption~\ref{hyp:control} is the most critical; Assumptions~\ref{hyp:outcome} plus \ref{hyp:noises} mainly serve to simplify the analysis (see Remark~\ref{rem:noises}). Notably, if Assumptions~\ref{hyp:outcome} and \ref{hyp:noises} hold but not Assumption~\ref{hyp:control}, then $\P(X_{T=t} \neq X) > 0$ for some $t \in \T$, and it follows from Equation~\eqref{eq:compare} that the invariance $\law{Y_t} = \law{Y_{T=t}}$ for any $t \in \T$ basically occurs in the useless scenarios where $f_Y$ does not really change in response to its $X$-input. Therefore, excluding these pathological cases, Assumption~\ref{hyp:control} through Theorem~\ref{thm:identification} fully classifies single-outcome equivalence between an SCM satisfying Assumptions~\ref{hyp:outcome} plus \ref{hyp:noises} and an RCM under the standard causal-inference regime.\footnote{A similar \say{almost-true} equivalence between a graphical condition like Assumption~\ref{hyp:control} and its structural implications was already discussed in \citep[Section 7.4]{pearl2009causality}, where the so-called graphical and counterfactual criteria of exogeneity coincide after excluding incidental cases.} We now turn to contextualizing this abstract comparison within real-world problems.

\subsection{Interpretation}\label{sec:interpretation}

Section~\ref{sec:comparison} explained under which mathematical conditions potential outcomes $(Y_t)_{t \in \T}$ in the causal-inference setting and structural counterfactuals $(Y_{T=t})_{t \in \T}$ define---or not---the same single-outcome causal estimands. In this section, we discuss these formal results from practical and conceptual perspectives, and exemplify the meaning of reasoning counterfactually with nonequivalent models.

\subsubsection{A practical perspective on cases of (non)equivalence}\label{sec:nonequi}

Up until now, we adopted an analytical viewpoint where the models where abstract mathematical objects. However, an SCM is not a mere detached tool; it is meant to genuinely capture the world's functioning. As such, practitioners do not decide the graphical relationships between endogenous variables themselves, they are imposed by Nature. In contrast, an RCM introduces hypothetical counterfactual variables whose interpretation may depend on the assumptions placed upon them. This critically signifies that single-outcome equivalence between the \emph{true} latent SCM and a potential-outcome model satisfying conditional ignorability does not generically holds in many real-world problems, due to the dichotomy ruled by Assumption~\ref{hyp:control}.

To clarify this point, consider a problem where one relies on an RCM $\mathcal{R} \in \mathfrak{R}_\O$ equipped with positivity and conditional ignorability to apply standard causal-inference techniques. They wonder whether this model is equivalent to the underlying SCM $\M \in \mathfrak{M}_\O$, which is supposed to meet Assumptions~\ref{hyp:outcome} and \ref{hyp:noises}. Answering this question amounts to translating the pivotal causal-ordering assumption between $T$ and $X$ (Assumption~\ref{hyp:control}) and its negation into common language. 

\begin{figure}[t]
    \centering
    \begin{subfigure}[b]{0.32\textwidth}
        \centering
        \begin{tikzpicture}[-latex ,
        state/.style ={circle ,top color =white,
        draw}]
        \node[state] (X){$X$};
        \node[state] (Y) [below right =of X] {$Y$};
        \node[state] (S) [below left =of X] {$T$};
        \path (X) edge (S);
        \path (S) edge (Y);
        \path (X) edge (Y);
        \end{tikzpicture}
        \caption{No posttreatment covariate}
        \label{fig:exogenous_covariates}
     \end{subfigure}
     \hfill
     \begin{subfigure}[b]{0.32\textwidth}
        \centering
        \begin{tikzpicture}[-latex ,
        state/.style ={circle ,top color =white,
        draw}]
        \node[state] (Xo){$X_{\operatorname{en}}$};
        \node[state] (Xi) [above =of Xo] {$X_{\operatorname{ex}}$};
        \node[state] (Y) [below right =of Xo] {$Y$};
        \node[state] (S) [below left =of Xo] {$T$};
        \path (S) edge (Xo);
        \path (Xi) edge (S);
        \path (S) edge (Y);
        \path (Xi) edge (Y);
        \path (Xo) edge (Y);
        \path[-] (Xo) edge (Xi);
        \end{tikzpicture}
        \caption{Pre and posttreatment covariates}
        \label{fig:nonexogenous_covariates}
     \end{subfigure}
    \hfill
     \begin{subfigure}[b]{0.32\textwidth}
        \centering
        \begin{tikzpicture}[-latex ,
        state/.style ={circle ,top color =white,
        draw}]
        \node[state] (X){$X$};
        \node[state] (Y) [below right =of X] {$Y$};
        \node[state] (S) [below left =of X] {$T$};
        \path (S) edge (X);
        \path (S) edge (Y);
        \path (X) edge (Y);
        \end{tikzpicture}
        \caption{No pretreatment covariate}
        \label{fig:exogenous_treatment}
     \end{subfigure}
    \caption{Three possible configurations of the treatment in $\M$ under Assumption~\ref{hyp:outcome}. $X_{\operatorname{ex}} := X_{\operatorname{Endo}(T)}$ denotes the parents of $T$ in $X$ while $X_{\operatorname{en}}$ are the remaining covariates. Exogenous variables are not represented but Assumption~\ref{hyp:noises} holds. A single node can represent several variables. In (a), $T$ does not impact $X$ (Assumption~\ref{hyp:control}); in (b), $T$ may impact some $X$-variables and some $X$-variables may impact $T$; in (c), $X$ does not impact $T$.}
    \label{fig:decisive}
\end{figure}

Assumption~\ref{hyp:control} basically signifies that the covariates are not altered by the treatment, as illustrated in Figure~\ref{fig:exogenous_covariates}. Notably, this configuration encompasses various typical causal-inference scenarios: in clinical trials, the covariates $X$ (sometimes called \emph{pretreatment} variables) may influence the treatment allocation $T$ but never the contrary. In these common situations, both the RCM and the SCM produce the same single-outcome counterfactuals due to Corollary~\ref{cor:equivalence}. In particular, $\E[Y_1 - Y_0 \mid X=x] = \E[Y_{T=1} - Y_{T=0} \mid X=x]$. We crucially remind that equivalence does not necessarily hold at the cross-outcome level (as also stated in Corollary~\ref{cor:equivalence}).

Importantly, the negation of Assumption~\ref{hyp:control}---which states that $T$ is a parent of several (or even all) covariates in $\mathcal{M}$---is also mathematically possible. Figures~\ref{fig:nonexogenous_covariates}~and~\ref{fig:exogenous_treatment} illustrate the possible causal graphs. In these situations, Theorem~\ref{thm:identification} does not guarantee single-outcome equivalence. Consequently, confusion between the two causal approaches can lead to misleading results: according to Equation~\eqref{eq:compare}, the RCM considers counterfactual outcomes at fixed $X$, whereas the SCM alters the covariates into $X_{T=t}$. These cases are empirically relevant, since people also rely on causal inference outside the scope of clinical trials, in settings where the treatment impacts the covariates. For example, $T$ drives $X$ but not the contrary (as in Figure~\ref{fig:exogenous_treatment}) in emblematic causal problems such as the Berkeley's admission paradox where $T$ represents the sex and $X$ the course choice \citep{bickel1975sex}. This is more generally true in the causal-fairness literature, where the variable to alter typically encodes an intrinsic feature like the sex, the race, or the age of individuals (see for instance \citep{kusner2017counterfactual,chiappa2019causal,plecko2020fair,nilforoshan2022causal,barocas-hardt-narayanan} for machine-learning-related research). Other scenarios, notably motivated by epidemiological research, include both pre and posttreatment covariates (as in Figure~\ref{fig:nonexogenous_covariates}) \citep{vanderweele2014effect}. In these cases, the treatment often represents an unstable characteristic of individuals like a habit (smoking) or a fluctuating biological factor (obesity). Section~\ref{sec:illustration} and Appendix~\ref{sec:middle_case} provide detailed illustrations of real-world-inspired problems where Assumption~\ref{hyp:control} does not hold. 

\begin{remark}[No causation without manipulation]\label{rem:holland}
The \say{no causation without manipulation} principle of Holland \cite{holland1986statistics} states that the causal influence of immutable variables like a person's race is ill-defined since we cannot conceptualize an assignment mechanism for such a nonmanipulable status (in contrast, for example, to allocating a medical treatment). There is a long-standing debate among causality scholars on this topic. Some agree that this principle should be a general rule \citep{winship1999estimation,freedman2004graphical,berk2004regression,greiner2011causal}, other argue that addressing nonmanipulable causes is scientifically relevant and proposed interpretations of this practice \citep{vanderweele2014causal,glymour2017evaluating,pearl2018does,pearl2019do}. This question relates to the current discussion around Assumption~\ref{hyp:control}, illustrated by Figure~\ref{fig:decisive}, at two levels.

The first level is independent of the employed causal models. Manipulability is an experimental concept that does not have a definite formal translation. In contrast, Assumption~\ref{hyp:control} is a formal condition on the causal ordering between endogenous variables. As such manipulability and causal ordering do not overlap, but they may intersect when approaching problems from a real-world standpoint, since manipulable and immutable causes usually conform to different graphical configurations in practice. As aforementioned, a \emph{manipulable} medical treatment commonly corresponds to Figure~\ref{fig:exogenous_covariates}, an \emph{immutable} attribute like race commonly corresponds to Figure~\ref{fig:exogenous_treatment}, an \emph{unstable} characteristic like being a smoker commonly corresponds to Figure~\ref{fig:nonexogenous_covariates}.\footnote{It is usual but not necessary: an arguably nonmanipulable feature like height can be positioned as $T$ in Figure~\ref{fig:exogenous_covariates} for being a responder of sex; an arguably manipulable feature like chocolate consumption can be placed as $T$ in Figure~\ref{fig:exogenous_covariates} for having no determinant among the other studied endogenous variables.} As such, some people consider that the \say{no causation without manipulation} maxim implies that the covariates must be fixed, pretreatment variables only \citep[Section 4.1]{imbens2020potential}. According to someone who follows this implication, it does not make sense to tackle problems corresponding to Figures~\ref{fig:nonexogenous_covariates} and \ref{fig:exogenous_treatment}, regardless of whether they work with an RCM or an SCM.

The second level opposes RCMs and SCMs. From a purely mathematical perspective, both causal approaches accept immutable causes: an RCM poses no formal constraints on its treatment variable, and any endogenous variable of an SCM---whatever its graphical position---is eligible to do-interventions. However, RCM analysts have traditionally precluded pretreatment covariates from their models \citep{imbens2015causal}, while SCM analysts have frequently applied do-interventions on variables with no endogenous parents \citep{kusner2017counterfactual,chiappa2019path, plecko2020fair}. According to someone who deems that \say{no posttreatment covariate} specifically constrains RCMs, it does not make sense to compare (as we did) an RCM with the latent SCM in the settings of Figures~\ref{fig:nonexogenous_covariates} and \ref{fig:exogenous_treatment}.

We emphasize that our work does not focus on philosophical arguments: it neutrally analyzes the literature. In particular, we do not contribute to the general debate on manipulability and do not argue in favor or against including implications of Holland's principle specifically in Rubin's framework. From this angle, we justify studying situations with posttreatment covariates and even comparing RCMs and SCMs in such cases for theoretical and empirical reasons. Firstly, as aforementioned, there is nothing that mathematically forbids to study these cases, be it with an RCM or an RCM. Therefore, addressing them is necessary to completely compare the two causal approaches. Secondly, many researchers have actually relied on (or referred to) the potential-outcome framework to understand the influence of immutable variables like sex, race, or biological factors \citep{li2017discrimination, glymour2017evaluating, khademi2019fairness, khademi2020algorithmic, qureshi2020causal}, typically placed like the treatment in Figure~\ref{fig:exogenous_treatment}.\footnote{Other articles, for instance \citep{bertrand2004emily,ridgeway2006assessing,gaebler2022causal}, explicitly focus on sex and race as \emph{perceived} by a decider. Such a perception could depend on the covariates and not be immutable.} Moreover, some references have presented side by side Rubin's approach and Pearl's approach specifically in such contexts, notably in fairness problems \citep{barocas-hardt-narayanan,makhlouf2024causality}. If we consider this corpus of the causal-inference literature to be admissible, then for the sake of empirical relevance we must also compare RCMs and SCMs in cases possibly disobeying \say{no causation without manipulation} or Assumption~\ref{hyp:control}.
\end{remark}

All in all, under the fundamental assumptions of causal inference, equivalence of single-outcome counterfactuals across the two causal models depends on the relationships between the treatment and the covariates, as described in Figure~\ref{fig:decisive}. What typically distinguishes the different configurations is the nature of the so-called treatment. If the treatment can be assigned \emph{a posteriori} (as in Figure~\ref{fig:exogenous_covariates}), then the two notions of counterfactuals coincide. Otherwise (as in Figures~\ref{fig:nonexogenous_covariates} and \ref{fig:exogenous_treatment}) structural counterfactuals and potential-outcome counterfactuals are generically not equal. The concrete example from Section~\ref{sec:ex_intro}, where the SCM fits Figure~\ref{fig:exogenous_treatment} and the potential outcomes of the RCM produce different causal estimands than structural counterfactuals, epitomizes this point.

\subsubsection{A conceptual perspective on cases of (non)equivalence}\label{sec:concept}

We now advance from the practical standpoint to a conceptual perspective. The fact that $\mathcal{R}$ and $\mathcal{R}_{\M}$ are not necessarily single-outcome equivalent under the assumptions of Theorem~\ref{thm:identification} can be interpreted as follows: $(Y_t)_{t \in \T}$ meeting conditional ignorability and $(Y_{T=t})_{t \in \T}$ obtained by do-interventions of $T$ do not carry the same counterfactual semantic. Said differently, their \say{had $T$ been equal to $t$} do not mean the same thing. This can similarly be understood as: $(Y_t)_{t \in \T}$ results from formally distinct interventions than \emph{do}-interventions on $T$. Beyond simply observing that these two approaches may produce distinct outcomes, understanding their respective significations is crucial, since practitioners must be able to explain and justify the estimands they aim to compute. We address this point by reminding some basics of counterfactual reasoning in a broad sense.

Counterfactual reasoning can be defined as thinking about outcomes in hypothetical worlds where some circumstances change from what factually happened while others are kept equal. Crucially, there is not a single way of reasoning counterfactually.\footnote{This nonuniqueness can be appreciated through the seminal work of Lewis on counterfactual conditionals \citep{lewis1973causation, lewis1973counterfactuals,lewis1979counterfactual}. According to Lewis, the verification of counterfactual statements depends on the arbitrary choice of a similarity relation.}  Each way notably depends on what is kept equal across worlds. Theorem~\ref{thm:identification} outlines the difference in this regard between the potential-outcome framework under the fundamental assumptions of causal inference and the structural account of counterfactuals. Recall that under the required assumptions, for every $t \in \T$
\begin{align*}
    &Y \aseq f_Y(T,X,U_Y) \aseq f_Y(T,f_X(T,U_X),U_Y),\\
    &Y_{T=t} \aseq f_Y(t,X_{T=t},U_Y) \aseq f_Y(t,f_X(t,U_X),U_Y),\\
    &\law{(T,X,Y_t)} = \law{(T,X,f_Y(t,X,U_Y)}. 
\end{align*}
Compared to the factual outcome $Y$, structural counterfactuals differ in $T$ but share the same background factors $(U_X,U_Y)$ in an almost-sure sense, while potential outcomes differ in $T$ but share the same input variables $(X,U_Y)$ in a distributional sense. Said differently, up to measure-theoretic technicalities, potential-outcome counterfactuals are \emph{ceteris paribus} counterfactuals (\emph{i.e.}, all other things being kept equal) with respect to all input variables---notably the covariates---whereas structural counterfactuals are \emph{ceteris paribus} counterfactuals with respect to all background factors---but \emph{mutatis mutandis} (\emph{i.e.}, after changing what must be changed) with respect to the covariates. According to Corollary~\ref{cor:equivalence}, these two conceptions of counterfactuals \emph{happen} to coincide in settings like Figure~\ref{fig:exogenous_covariates} where there is no posttreatment covariate. But essentially, they are different.

We emphasize that both definitions of counterfactuals are perfectly legitimate and bear causal interpretations. However, they convey distinct meanings and thereby correspond to different causal effects. Therefore, \emph{they should not be employed for the same purpose}. Let us illustrate this aspect on a concrete case.

\subsubsection{Illustration: an immutable treatment and two different kinds of counterfactuals}\label{sec:illustration}

We conclude Section~\ref{sec:main} by exemplifying the mathematical, practical, and conceptual considerations of the (non)equivalence results from Section~\ref{sec:theorem}. To do so, we compare RCMs and SCMs under the assumptions of Theorem~1 in concrete examples that do not satisfy Assumption~\ref{hyp:control} so that single-outcome equivalence does not hold. More precisely, in what follows we further study the motivating example from Section~\ref{sec:ex_intro} which corresponds to Figure~\ref{fig:exogenous_treatment}. For concision, we defer to Appendix~\ref{sec:middle_case} the analysis of a scenario adapted to Figure~\ref{fig:nonexogenous_covariates}.

We firstly remind the example's setting while properly defining the considered causal models. The treatment status $T$ indicates the gender; the covariate $X$ quantifies the level of work experience; the outcome $Y$ evaluates a candidate's application for some position. Suppose that the associated observational vector $\O := (T,X,Y)$ is ruled by the following SCM $\M \in \mathfrak{M}_\O$:
\begin{align*}
    T &\aseq U_T,\\
    X &\aseq \alpha T + U_X,\\
    Y &\aseq X + \beta T + U_Y,
\end{align*}
where $\alpha$ and $\beta$ are deterministic parameters. This model satisfies Assumption~\ref{hyp:outcome} and the negation of Assumption~\ref{hyp:control} by design. Moreover, we suppose that positivity (Assumption~\ref{hyp:positivity}) is true (for instance by choosing $U_X$ inducing a Gaussian distribution), and that $U_Y \independent (U_T,U_X)$ so that Assumption~\ref{hyp:noises} holds. Finally, we set two potential outcomes $(Y_0,Y_1)$ meeting the consistency rule, that is $Y \aseq (1-T) \cdot Y_0 + T \cdot Y_1$, and single-outcome conditional ignorability (Assumption~\ref{hyp:single_ignorability}). This defines $\mathcal{R} := (T,X,Y_0,Y_1) \in \mathfrak{R}_\O$. Note that all the assumptions of Theorem~\ref{thm:identification} are satisfied, and that the SCM fits Figure~\ref{fig:exogenous_treatment}.

We recall the estimands that an analyst working with $\mathcal{R}$ and an analyst working with $\mathcal{R}_{\M}$ respectively obtain for the average causal effect of $T$ onto $Y$ conditional to $X=x$. The former analyst gets the following CATE based on potential outcomes:
\begin{align*}
    \operatorname{CATE}_{\mathcal{R}}(x) &:= \E[Y_1-Y_0 \mid X=x]\\
    &= \E[(X + \beta + U_Y) - (X + U_Y) \mid X=x]\\
    &= \beta.
\end{align*}
This time, we applied Theorem~\ref{thm:identification} in the first step of the above calculus to illustrate that it gives the same results as using Lemma~\ref{lm:sw} plus $U_Y \independent (T,X)$ like in Section~\ref{sec:ex_intro}. The latter analyst gets the following CATE based on structural counterfactuals:
\begin{align*}
    \operatorname{CATE}_{\mathcal{R}_{\M}}(x) &:= \E[Y_{T=1}-Y_{T=0} \mid X=x]\\
    &= \E[(\alpha + U_X + \beta + U_Y) - (U_X + U_Y) \mid X=x]\\
    &= \alpha + \beta.
\end{align*}
As noticed in Section~\ref{sec:ex_intro}, $\operatorname{CATE}_{\mathcal{R}} \neq \operatorname{CATE}_{\mathcal{R}_{\M}}$ if $\alpha \neq 0$. Because these estimands are single-outcome effects, we conclude that $\mathcal{R}$ and $\M$ are not single-outcome equivalent, as expected by the graphical discussion from Section~\ref{sec:nonequi}. Similar conclusions hold in settings with pre and posttreatment covariates (Figure~\ref{fig:nonexogenous_covariates}), as demonstrated in Appendix~\ref{sec:middle_case}.

Let us interpret in more details the obtained CATEs through the conceptual lens of Section~\ref{sec:concept}. Observe that $\operatorname{CATE}_{\mathcal{R}}$ measures only the \emph{direct effect} of the treatment in $\M$: it completely ignores the dependence of $Y$ on $T$ through $X$, as it involves only $\beta$. In contrast, remark that $\operatorname{CATE}_{\mathcal{R}_{\M}}$ measures the \emph{total effect} of the treatment: it takes into account the whole path of influence of $T$ onto $Y$, involving both $\alpha$ and $\beta$. This is due to the \emph{ceteris paribus/mutatis mutandis} significations of the associated counterfactual variables: by holding $X$ fixed, $(Y_0,Y_1)$ blocks information that does not flow \emph{directly} from $T$; by letting $X$ change, $(Y_{T=0},Y_{T=1})$ captures also information that flows \emph{indirectly} from $T$ through $X$. From a fairness perspective, $\operatorname{CATE}_{\mathcal{R}}$ says that if $\beta = 0$, that is if $T$ is not a \emph{direct} cause of $Y$, then the application process if fair; whether it is unfair towards men or women when $\beta \neq 0$ depends on the sign of $\beta$. In contrast, $\operatorname{CATE}_{\mathcal{R}_{\M}}$ says that if $\beta = -\alpha$, that is if the decision rule $Y$ compensates the discrepancy of work experiences $X$ across genders $T$, then the application process is fair. Each analysis points out a different notion of fairness: considering $\operatorname{CATE}_{\mathcal{R}_{\M}}$ as a fairness criterion suggests that recruiters should correct societal inequalities by preferring women with potentially lower work experience but higher merit whereas relying on the $\operatorname{CATE}_{\mathcal{R}}$ suggests it is only explicitly including the gender in the decision-rule pipeline that is unfair.\footnote{Such a differentiation is closely related to the notion of \emph{path-specific counterfactual fairness} \citep{chiappa2019path}.} Critically, if $\alpha \neq 0$, \emph{practitioners mixing potential outcomes with structural counterfactuals could reach contradictory conclusions on fairness.} This is why having a clear understanding of the semantic carried by the employed counterfactual outcomes is crucial. We conclude the section with additional remarks raised by this example.

\begin{remark}[About the structural assumptions]\label{rem:noises} At the end of Section~\ref{sec:theorem}, we argued that the equivalence between an RCM satisfying the fundamental assumptions of causal inference and the latent SCM was mostly governed by Assumption~\ref{hyp:control} (no posttreatment covariate), while Assumptions~\ref{hyp:outcome} and \ref{hyp:noises} mainly served to derive general results. Let us illustrate this point on the above example.

To emphasize the respective roles of Assumptions~\ref{hyp:outcome}, \ref{hyp:noises} and \ref{hyp:control}, we consider the same setting but we do not assume $U_Y \independent (U_X,U_T)$ anymore. Consequently, Assumption~\ref{hyp:noises} does not necessarily hold and $U_Y$ may depend on $(T,X)$. This relaxation does not affect $\operatorname{CATE}_{\mathcal{R}_{\M}}$ which remains equal to $\alpha + \beta$. As Theorem~\ref{thm:identification} cannot be employed, we compute $\operatorname{CATE}_{\mathcal{R}}$ via Lemma~\ref{lm:sw}. This gives:
\begin{align*}
    \operatorname{CATE}_{\mathcal{R}}(x) &:= \E[Y_1-Y_0 \mid X=x]\\
    &= \E[Y \mid X=x, T=1] - \E[Y \mid X=x, T=0]\\
    &= \E[x + \beta + U_Y \mid X=x, T=1] - \E[x + U_Y \mid X=x, T=0]\\
    &= \beta + \E[U_Y \mid X=x, T=1] - \E[U_Y \mid X=x, T=0].
\end{align*}
In comparison to before, the difference $\E[U_Y \mid X=x, T=1] - \E[U_Y \mid X=x, T=0]$ can be nonzero. Remark that, except in the very specific case where $\E[U_Y \mid X=x, T=1] - \E[U_Y \mid X=x, T=0] = \alpha$ for $\law{X}$-almost every $x$, we still have $\operatorname{CATE}_{\mathcal{R}} \neq \operatorname{CATE}_{\mathcal{R}_{\M}}$. This means that $\mathcal{R}$ and $\mathcal{R}_{\M}$ are not single-outcome equivalent when even when Assumptions~\ref{hyp:outcome} and \ref{hyp:noises} do not hold.
\end{remark}

\begin{remark}[Computing direct effects from an SCM]\label{rem:effect}
One can still define and compute $\operatorname{CATE}_{\mathcal{R}}$, that is the \emph{direct} effect, using an SCM $\M$. We propose two different expressions of this quantity with \emph{structural} terms in the setup of the fairness illustration,
\begin{align}
    \operatorname{CATE}_{\mathcal{R}}(x) &= \E[(x + \beta + U_Y) - (x + U_Y)],\nonumber\\
    &= \E[Y_{T=1,X=x} - Y_{T=0,X=x}],\label{eq:cf_expr}\\
    &= \E[Y_{T=1} \mid X_{T=1}=x] - \E[Y_{T=0} \mid X_{T=0}=x]\label{eq:int_expr}.
\end{align}
Leveraging $\mathcal{R}$ instead of $\M$ permits to compute the direct effect without specifying $\M$. In contrast, one cannot always define and compute the \emph{total} effect $\operatorname{CATE}_{\mathcal{R}_{\M}}$ from $\mathcal{R}$. More generally, by expressing the single-outcome expression of $\mathcal{R}$ in terms of $\M$, Theorem~\ref{thm:identification} enables one to find SCM-based formulas of effects derived from $\mathcal{R}$ under the fundamental assumptions of causal inference.

This remark also serves to precise a key message of the paper. Note that Equation~\eqref{eq:cf_expr} employs interventions on $(T,X)$, while Equation~\eqref{eq:int_expr} employs interventions solely on $T$ but involves the post-intervention covariates. Critically, $\operatorname{CATE}_{\mathcal{R}}$ cannot be written with $(T,X,Y_{T=0},Y_{T=1})$ only. This means that $\operatorname{CATE}_{\mathcal{R}}$ is identifiable in $\M$ but \emph{not} in $\mathcal{R}_{\M}$. Therefore, one should not conclude from our results that the effects derived from an RCM $\mathcal{R}$ in the causal-inference setting are in general incompatible with the ones derived from the latent SCM $\M$. What we show is that swapping the respective counterfactual-outcome variables relative to $T$ of these models, which amounts to exchanging $\mathcal{R}$ and $\mathcal{R}_{\M}$, generally defines distinct estimands. In other words, our work highlights a difference between potential outcomes meeting classical assumptions and do-interventions on $T$ in $\M$, rather than between RCMs and SCMs (see also Remark~\ref{rem:intervention}).
\end{remark}

To sum-up, each each type of counterfactual variables has a different causal signification in this immutable-treatment configuration, and therefore produces different causal estimands. This signifies that the difference between frameworks does not amount to methodological considerations only. Analysts and researchers should also justify the chosen model and its assumptions depending on the kind of causal effects they want to compute. In the final section, we include the conceptual and practical divergences between models pointed out by this illustration into a more comprehensive reflection.  

\section{Discussion}\label{sec:consequences}

This section examines common practices in the causal-inference literature in light of the mathematical differences between models explained in Section~\ref{sec:main}. Firstly, Section~\ref{sec:choice} clarifies the relation between our contribution and the notorious formal equivalence between frameworks, unveiling a fundamental dichotomy in the applications of potential outcomes. Secondly, Section~\ref{sec:exchange} leverages this discussion to provide recommendations on the exchange of potential-outcome and do notations.

\subsection{On the formal equivalence between frameworks}\label{sec:choice}

As mentioned in the introduction, many articles interchangeably use the potential-outcome notation and the do notation, invoking an \say{equivalence} between causal frameworks. This may seem paradoxical after reading Section~\ref{sec:main}. In the following, we clarify this aspect. Then, we explain what a conflict between models like Section~\ref{sec:illustration} signifies to the relationship between SCMs and RCMs, and to how people generally manipulate these models.

\subsubsection{Structural representation and graphical translation}

Recall that every $\M \in \mathfrak{M}_\O$ entails an RCM $\mathcal{R}_\M \in \mathfrak{R}_\O$ according to Lemma~\ref{lem:consistency_do}. Conversely, any $\mathcal{R} \in \mathfrak{R}_\O$ can be \emph{represented} by an SCM.
\begin{proposition}[Structural representation]\label{prop:existence}
Let $\O$ be an observational vector. For any $\mathcal{R} \in \mathfrak{R}_\O$, there exists $\M_{\mathcal{R}} \in \mathfrak{M}_\O$ such that $\M_{\mathcal{R}}$ and $\mathcal{R}$ are almost-surely equivalent. We say that $\M_{\mathcal{R}}$ is a \emph{structural representation} of $\mathcal{R}$.
\end{proposition}
This result can be seen as a variant of \citep[Proposition 1]{ibeling2024comparing} in our specific setting. The expression \emph{structural representation} is inspired from it.

Interestingly, if one chooses to define potential outcomes as structural counterfactuals from a structural representation, then Rubin's causal framework and Pearl's causal framework become two different languages to talk about the same objects. In the RCM, assumptions for causal inference are generally framed as conditional-independence restrictions on counterfactual variables (like Assumptions~\ref{hyp:cross_ignorability} and \ref{hyp:single_ignorability}); in Pearl's causal framework, assumptions on causal relationships are generally framed in terms of graphical conditions on factual variables (like Assumptions~\ref{hyp:outcome}, \ref{hyp:noises} and \ref{hyp:control}). Both \citep[Chapter 7]{pearl2009causality} and \citep{richardson2013single} focus on unifying these two mathematical languages by providing rules for translating assumptions and theorems from one viewpoint to the other. This ensures what people often refer as the \emph{logical} or \emph{formal equivalence} between frameworks. We crucially emphasize that, in contrast to the notions in Definition~\ref{def:representation}, this corresponds to an equivalence between formalisms---not between given causal models. This formal equivalence notably allows analysts to work symbiotically with an RCM $\mathcal{R}$ and a structural representation $\M_{\mathcal{R}}$, or with an SCM $\M$ and its entailed RCM $\mathcal{R}_\M$.

\begin{remark}[Graphical interpretation of conditional ignorability]\label{rem:backdoor}
A classic translation rule is \citep[Theorem 4.3.1]{pearl2016causal}, which states that if $X$ meets the \emph{backdoor criterion} relative to $(T,Y)$ in $\M$ \citep[Definition 3.3.1]{pearl2009causality}, then $Y_{T=t} \independent T \mid X$ for all $t \in \T$: said differently, single-outcome conditional ignorability holds in $\mathcal{R}_\M$. This counterfactual interpretation of graphical conditions echoes Corollary~\ref{cor:equivalence}. Our result states that if an SCM satisfies Assumptions~\ref{hyp:outcome}, \ref{hyp:noises} and \ref{hyp:control}, then it is single-outcome equivalent to an RCM such that $Y_t \independent T \mid X$ for all $t \in \T$ (Assumption~\ref{hyp:single_ignorability}). The connection between \citep[Theorem 4.3.1]{pearl2016causal} and Corollary~\ref{cor:equivalence} follows from the fact that Assumptions~\ref{hyp:outcome}, \ref{hyp:noises} and \ref{hyp:control} (illustrated by Figure~\ref{fig:exogenous_covariates}) entail that $X$ meets the backdoor criterion relative to $(T,Y)$ in $\M$.

Interestingly, the two results differ from their perspectives. Corollary~\ref{cor:equivalence} reframes Pearl's translation rule in terms of equivalence between presumably \emph{separated} SCMs and RCMs. This change of viewpoint helps understanding some unappreciated aspects of the formal equivalence, as discussed below. 
\end{remark}

\subsubsection{Two aspects of the formal equivalence}

There is no logical contradiction between the results from Section~\ref{sec:main} and this formal equivalence. Nevertheless, our contribution highlights two overlooked features.

Firstly, Proposition~\ref{prop:existence} or \citep[Proposition 1]{ibeling2024comparing} simply \emph{allow} to represent an RCM by an SCM. Ultimately, defining potential outcomes as structural counterfactuals is a \emph{choice}---it does not rest on any proof. When Pearl writes $Y_t := Y_{T=t}$ in \citep[Equation 3.51]{pearl2009causality}, claiming that the operation $\operatorname{do}(T=t)$ on the SCM gives a physical meaning to the vague \say{had $T$ been $t$} of the potential outcome, this is an arbitrary choice. As demonstrated by Propositions~\ref{prop:nas} and \ref{prop:nas_fa}, naively looking at the definitions of the Rubin's causal framework and of Pearl's causal framework, there is nothing that mathematically constrains potential outcomes to coincide at any level with the structural counterfactuals of a chosen SCM.\footnote{From a more philosophical angle, \citep[Section 2.1]{markus2021causal} made a similar remark to argue that the two frameworks were \emph{weakly} equivalent rather than \emph{strongly} equivalent.} What \citep[Chapter 7]{pearl2009causality} and \citep{richardson2013single} really show is that \emph{if} potential outcomes are chosen to be structural counterfactuals, then one can translate the assumptions made on an underlying causal graph into assumptions on potential outcomes---not that they \emph{must} be chosen as such.

Secondly, even though there always exist \emph{theoretical} structural representations of a given RCM $\mathcal{R}$, nothing guarantees that they correctly capture the real-world causal dependencies. In other words, the true SCM $\M$ may not represent $\mathcal{R}$. Notably, Section~\ref{sec:nonequi} specifies situations where an RCM $\mathcal{R}$ tailored to causal inference is not single-outcome equivalent to the true SCM $\M$ (as illustrated in Section~\ref{sec:illustration}), which implies that it cannot be represented by $\M$.

The fact that (1) using the formal equivalence is a choice and (2) applying the formal equivalence on a predefined RCM may produce an SCM conflicting with the true one has important consequences on its practicality that we discuss next. 

\subsubsection{Two fundamentally distinct paradigms for potential outcomes}

An SCM is meant to correctly capture the cause-effect links between the variables of interest. This is why no analyst deliberately works with an SCM that is apparently wrong. In particular, it does not make sense to use an RCM $\mathcal{R}$ in synergy with a structural representation if the true SCM $\M$ itself does not represent $\mathcal{R}$. Who would axiomatize potential outcomes via do-interventions in a fake SCM? Under this principle, two paradigms for defining and applying a potential-outcome model $\mathcal{R}$ coexist.

On the one hand, Pearl has firmly advocated for long to \emph{always} use the potential-outcome framework in symbiosis with an SCM. This rule induces a first paradigm which amounts to defining the RCM as $\mathcal{R} := \mathcal{R}_\M$ where $\M$ is the \emph{true} SCM. This means letting $(Y_t)_{t \in \T} = (Y_{T=t})_{t \in \T}$. In this approach, properties of the potential outcomes (like conditional ignorability) do not come from fundamental assumptions but follow from hypotheses made on $\M$, generally framed as graphical conditions (recall Remark~\ref{rem:backdoor}). Note that this forbids the practice from Section~\ref{sec:illustration} where we employed $\mathcal{R}$ satisfying conditional ignorability whereas Assumption~\ref{hyp:control} did not hold in $\M$. According to Pearl, a key interest of this approach comes from the fact that conditional-independence assumptions of RCMs are hard to interpret, while the causal graph and structural equations of an SCM form an intelligible formalization of scientific knowledge from which such assumptions can be derived and justified. We refer to \citep[Section 7.4]{pearl2009causality} for his detailed argumentation. Sections~\ref{sec:concept} and \ref{sec:choice} underline that following this paradigm consists in making a specific choice regarding the hypothetical interventions defining the potential outcomes, thereby regarding their counterfactual semantics: it formally defines their \say{had $T$ been $t$} as the operation $\operatorname{do}(T=t)$ in $\M$ which changes the remaining endogenous variables accordingly. 

On the other hand, defining an RCM differently than by do-interventions on the true SCM is mathematically doable, as shown several times throughout Section~\ref{sec:main}. Adopting this second paradigm does not mean denying the formal equivalence between frameworks; it means not \emph{applying} it. In this approach, potential outcomes $(Y_t)_{t \in \T}$ are typically primitives of the model not necessarily derived from $\M$ (recall Remark~\ref{rem:versus} and Figure~\ref{fig:principle}): the analyst may directly place assumptions on $\mathcal{R}$ (like conditional ignorability) regardless of what holds in $\mathcal{R}_\M$. But they can also be derivatives of $\M$ (or of another type of models) by other actions than do-interventions (see Remark~\ref{rem:intervention}). What characterizes the paradigm is essentially not a matter of \say{primitive versus derivative}, rather the fact that potential outcomes are not axiomatized by do-interventions on $T$ in $\M$. This construction implies that the structural representations of $\mathcal{R}$ may differ from $\M$ and should thereby be ignored. The chosen (or derived) assumptions define the semantics of the counterfactual statements associated to $\mathcal{R}$. Recall that Section~\ref{sec:concept} notably showed that dressing potential outcomes with conditional ignorability defines their \say{had $T$ been $t$} as switching $T$ into $t$ while keeping all remaining variables unchanged (at the single-outcome level). Other assumptions could give other semantics. A methodological interest of this paradigm notably comes from the possibility to compute the \emph{direct} causal effects of treatments breaking Assumption~\ref{hyp:control} through statistical methods without knowing $\M$ (as explained in Remark~\ref{rem:effect}). This practice is perhaps controversial and philosophically debatable. Nevertheless, we point out that (even though it is generally implicitly done) \emph{not} defining potential outcomes as the true structural counterfactuals is actually something common in the scientific literature. In Remark~\ref{rem:holland}, we mentioned a series of articles leveraging the potential-outcome framework equipped with the fundamental assumptions of causal inference to compute the effects of immutable characteristics \citep{li2017discrimination, glymour2017evaluating, khademi2019fairness, khademi2020algorithmic, qureshi2020causal, makhlouf2024causality}. As explained in Section~\ref{sec:interpretation}, in such cases $\mathcal{R}$ is generally not even single-world equivalent to $\mathcal{R}_{\M}$. This positions these works in the second paradigms. Therefore, if we acknowledge them, then we must accept that unifying potential outcomes and structural counterfactuals is not an obligation.

\begin{remark}[About Holland's principle in the second paradigm]
Consider a practitioner who thinks that \say{no causation without manipulation} applies to RCMs but not to SCMs, and thereby only follows the second paradigm with no posttreatment covariate (Assumption~\ref{hyp:control}). To them, Corollary~\ref{cor:equivalence} ensures that whenever the potential-outcome framework can be applied, the employed RCM is at least single-outcome equivalent to the latent SCM under Assumptions~\ref{hyp:outcome} and \ref{hyp:noises}. This appear to rule out situations where potential outcomes may not coincide with structural counterfactuals in the second paradigm, such as the one from Sections~\ref{sec:ex_intro} and \ref{sec:illustration} (and the one from Appendix~\ref{sec:middle_case}). Let us comment on this. Firstly, even in cases like Figure~\ref{fig:exogenous_covariates}, cross-world or almost-sure equivalence is not guaranteed. This means that substituting $\mathcal{R}$ by $\mathcal{R}_{\M}$ is generally incorrect at stronger counterfactual-reasoning levels. This is why such a practitioner should still be careful. Secondly, as mentioned above, addressing a treatment not satisfying Assumption~\ref{hyp:control} with an RCM meeting the fundamental assumptions of causal inference is both theoretically possible and empirically done. This is why we discuss the implications of such a choice.
\end{remark}

\begin{remark}[About interventions and the true SCM in the second paradigm]\label{rem:intervention}
We emphasize that adopting the second paradigm does not mean abandoning the \emph{true} SCM $\M$; it means ignoring the \emph{structural representations} of the RCM. The sole consequence of this approach is that the causal interpretation of $(Y_t)_{t \in \T}$ possibly differs from the one of $(Y_{T=t})_{t \in \T}$ (as exemplified in Section~\ref{sec:concept}). Said differently, the second paradigm marks a tension between potential outcomes and \emph{do-interventions} on $T$ in $\M$ rather than between potential outcomes and the \emph{model} $\M$ itself. This echoes the conclusion of Remark~\ref{rem:effect}.

Notably, one could define $(Y_t)_{t \in \T}$ through other types of interventions in $\M$ than do-interventions. In particular, Theorem~\ref{thm:identification} shows that the single-outcome laws of $\mathcal{R}$ under the fundamental assumption of causal inference do not necessarily result from a do-intervention on $T$ in $\M$, but rather from an action that fixes $T$ to $t$ only in the equations corresponding to $Y$ in $\M$. We could for instance define the entailed RCM of $\M$ via such an intervention to reconcile the two causal models at this level.
\end{remark}

In the motivating example from Section~\ref{sec:ex_intro}, analyst $\M$ and analyst $\mathcal{R}$ respectively adopt the first and second paradigms to define their counterfactual outcomes. As such, analyst $\M$ exploits the formal equivalence between frameworks---not analyst $\mathcal{R}$. It feels that Pearl's rule to always leverage an SCM as the axiomatic characterization of potential outcomes through do-interventions possibly made unclear the existence of these two paradigms. We emphasize that we do not discuss which paradigm people should follow. Instead, we neutrally classify what people \emph{actually} do, and explain the implicit meaning of these practices. All in all, each approach can be legitimate; \emph{what crucially matters is having a clear understanding of the produced counterfactuals and being transparent about the choices made}. If an analyst aims to compute counterfactual estimands carrying a \emph{mutatis mutandis} signification as defined by do-interventions, then they should explicitly mention that they follow the first paradigm, and specify the true SCM with at least the key graphical relationships. If an analyst aims to compute counterfactual estimands carrying another signification, then they should explicitly mention that they follow the second paradigm, and be clear about the counterfactual interpretation of their assumptions (\textit{e.g.}, \emph{ceteris paribus} via conditional ignorability). The following case must be avoided: an RCM analyst does not explicitly define their model via a partially specified SCM, misinterprets the formal equivalence and hence wrongly believes that their potential outcomes necessarily derive from do-interventions on the true SCM, therefore computes distinct causal estimands than what they expected. In this sense, we conclude this article by defending a more cautious use of notations.

\subsection{On the exchange of notations}\label{sec:exchange}

Following the conclusion of Section~\ref{sec:choice}, this subsection discusses what researchers and practitioners should do. More specifically, our suggestions do not concern the methods people use, but how they \emph{present} their approaches and their results. Concretely, we argue that exchanging the do notation and the potential-outcome notation should be done with greater care than what is commonly done in the literature. A notation is the identification of a mathematical object. Therefore, a same notation can be used for two differently-defined objects just in case they are mathematically equal. In the first paradigm, where the potential outcomes are chosen as the true structural counterfactuals, the notational exchange with the do is valid; not in the second paradigm. In what follows, we examine books and articles referring to the formal equivalence between causal frameworks in a confusing way, to underline the importance of specifying and justifying the adopted paradigm as well as using adequate notations.\footnote{We do not suggest that these references contain erroneous claims, only that the presentation of specific aspects can be misleading. Beyond that, we strongly recommend reading them.} 

For example, \cite[Section 5]{colnet2024causal} commences with a potential-outcome model, and then invoke the formal equivalence to substitute its notations by do notations. Because writing a \say{do} only makes sense when there is an SCM involved, this exchange implicitly engages a structural representation of their RCM. But recall that nothing guarantees that this structural representation corresponds at the single-outcome level to the true SCM. In their case, everything works fine precisely because only settings with no posttreatment covariate are considered, only single-outcome causal effects are studied, and their RCM satisfies conditional ignorability. However, they never explicitly state Assumption~\ref{hyp:control} and even less explain its crucial role. While this may seem harmless in practice, this could mislead people to wrongly believe that this equivalence holds in general: it would generally not hold with posttreatment covariates or at stronger counterfactual-reasoning levels. Even from a purely logical viewpoint, working in this favorable scenario should not obviate proper justification to why $\law{(T,X,Y_t)} = \law{(T,X,Y_{T=t})}$ for every $t \in \T$. We suggest two options to clarify \cite[Section 5]{colnet2024causal}. The first is specifying that the first paradigm is adopted to define the RCM, and explaining how its assumptions (conditional ignorability) are derived from the true SCM (backdoor criterion). The second is specifying that the second paradigm is adopted to define the RCM, and justifying why the RCM and the true SCM happen to be single-outcome equivalent (Corollary~\ref{cor:equivalence}). Replacing the potential-outcome notations of the RCM with do notations only makes sense in the first option.

The case of \citep[Chapter 5]{barocas-hardt-narayanan} and \cite{makhlouf2024causality} also demands caution. These surveys introduce the two causal frameworks specifically in the context of fairness with identical notations on counterfactual variables and suggest that the appropriate choice of framework is mostly a matter of methodological considerations because of the formal equivalence. Additionally, they present RCM-based inference techniques specifically relying on conditional ignorability. This feels pedagogically dangerous. As repeated, in settings with posttreatment covariates---fairness problems typically---one can generally not interchangeably manipulate the counterfactual outcomes of an RCM under the fundamental assumptions for causal inference and the ones of the \emph{true} SCM. Section~\ref{sec:illustration} notably exemplified how mixing the notations could lead to contradictory results. In such settings, SCM counterfactuals and RCM counterfactuals correspond to different paradigms. Leveraging both paradigms in a same work requires distinct notations to prevent confusion. This is particularly critical when the treatment may impact the covariates. 

These examples underline that the formal equivalence in itself is not a correct justification for exchanging do notations and potential-outcome notations in an RCM: the key argument is the paradigm followed to define the RCM. Writing a \say{do} exclusively makes sense in studies clearly following the first paradigm. Writing a potential-outcome notation is acceptable in both paradigms (since structural counterfactuals are potential outcomes) but distinct notations are needed when relying on both paradigms for a same problem.

To summarize Section~\ref{sec:consequences}, we think that the scope and implications of the formal equivalence between causal frameworks can be misleading. In particular, it does not mean that $\mathcal{L}((T,X,(Y_t)_{t \in \T})) = \mathcal{L}((T,X,(Y_{T=t})_{t \in \T}))$ in general; it means that such an exchange can hold \emph{if equivalent assumptions are made across models}. Supposing distinct axioms across models means giving distinct interpretations to their respective counterfactual outcomes, which thereby relate to distinct causal interventions and causal effects. This is why we recommend to present Rubin's causal models and Pearl's causal models as distinct mechanisms for reasoning counterfactually that coincide under specific assumptions and choices, rather than merely different perspectives.

\section{Conclusion}

In this paper, we superimposed Pearl's causal framework and Rubin's causal framework without presuppositions to show that structural counterfactual outcomes and potential outcomes do not necessarily coincide at any levels of counterfactual reasoning. To furnish a thorough comparison at the most relevant level, we expressed the laws of potential outcomes in terms of the latent SCM under classical causal-inference assumptions. On the basis of this result, we gave a detailed interpretation of counterfactuals in each causal framework, specifying when they entailed different conclusions. More specifically, counterfactual inference with potential outcomes under conditional ignorability yields \emph{ceteris paribus} counterfactuals with respect to the covariates, whereas counterfactual inference with a do-intervention on an SCM yields \emph{mutatis mutandis} counterfactuals with respect to the covariates. If the cause of interest affects the covariates, these constructions are generally not equal in law. For these reasons, we call the community to not interchangeably use the potential-outcome framework and do-interventions to define counterfactual outcomes, unless the justification is explicitly made.

We emphasize that our contribution is not an argument in favor of using one causal model rather than the other, or against the formal equivalence between frameworks. Instead of taking position or addressing philosophical arguments, it highlights some facts: theoretically, one can perfectly define potential outcomes as distinct to the true structural counterfactuals; empirically, researchers have actually (implicitly) worked with potential outcomes defined as distinct to the true structural counterfactuals. Our work is meant to shed light on the different mathematical choices that analysts can make when working with counterfactual outcomes, and to precise their implications in order to prevent incorrect or ambiguous conclusions in causal studies. In doing this paper, we hope to clarify the similarities and differences between the two major causal approaches.

\section*{Acknowledgment}
The author thanks Antoine Chambaz and \'Emilie Devijver for helpful discussions and valuable suggestions, as well as two anonymous reviewers for their useful recommendations. Most of this work has been done when the author was at Institut de Mathématiques de Toulouse (Université Paul Sabatier).

\section*{Author's statement}
This work was partially funded by the project CAUSALI-T-AI (ANR-23-PEIA-0007) of the French National Research Agency. The author confirms the sole responsibility for the presented results and manuscript preparation. The author states no conflict of interest.

\appendix

\section{Illustration with pre and posttreatment covariates}\label{sec:middle_case}

Section~\ref{sec:illustration} compared causal effects obtained with respectively potential outcomes under the causal-inference regime and structural counterfactuals in a problem corresponding to Figure~\ref{fig:exogenous_treatment}. This section makes the same comparison in the context of Figure~\ref{fig:nonexogenous_covariates}. It is meant to propose a real-world illustration of Figure~\ref{fig:nonexogenous_covariates}, and to prove that nonequivalence does not generally hold in this case as well.

The following example (more precisely its graph) is inspired by \cite{li2022deep}. The treatment status $T : \Omega \to \{0,1\}$ indicates whether an individual is a smoker or not; a first covariate $X_1 : \Omega \to \R$ represents the expression of a gene, which is higher when the gene is more active; a second covariate $X_2 : \Omega \to \R$ quantifies blood pressure; the outcome $Y : \Omega \to \R$ evaluates a person's health, a higher score meaning a better health. Suppose that these variables are observed in the context of a medical study. Causal analysts are task with assessing the effect of smoking onto health. Assume that the observational vector $\O := (T,X,Y)$ is ruled by the following SCM $\M \in \mathfrak{M}_\O$:
\begin{align*}
    &T \aseq \mathbf{1}_{\{X_1 + U_T>0\}},\\
    &X_1 \aseq U_1,\\
    &X_2 \aseq \alpha T + U_2,\\
    &Y \aseq \gamma X_1 - X_2 + \beta T + U_Y,
\end{align*}
where $\alpha, \beta, \gamma$ are again deterministic parameters. Typically, $\alpha > 0$ encodes that smoking increases blood pressure, $\beta < 0$ represents the direct negative impact of smoking onto health, and $\gamma < 0$ describes a gene that deteriorates health (and makes people more inclined to smoke). Using the notations of the paper: $X := (X_1,X_2)$ and $U_X := (U_1,U_2)$. As in the fairness illustration, the model satisfies Assumption~\ref{hyp:outcome} and the negation of Assumption~\ref{hyp:control} by design. Moreover, as before we suppose that positivity (Assumption~\ref{hyp:positivity}) is true and that $U_Y \independent (U_T,U_X)$ so that Assumption~\ref{hyp:noises} holds. Note that the SCM fits Figure~\ref{fig:nonexogenous_covariates}; the exact graph is more precisely depicted in Figure~\ref{fig:middle_case}. Finally, we set two potential outcomes $(Y_0,Y_1)$ meeting the consistency rule, that is $Y \aseq (1-T) \cdot Y_0 + T \cdot Y_1$, and single-outcome conditional ignorability (Assumption~\ref{hyp:single_ignorability}). This defines $\mathcal{R} := (T,X,Y_0,Y_1) \in \mathfrak{R}_\O$. All the assumptions of Theorem~\ref{thm:identification} are satisfied.

Let us compute the estimands that an analyst working with $\mathcal{R}$ and an analyst working with $\mathcal{R}_{\M}$ respectively obtain for the average causal effect of $T$ onto $Y$ conditional to $X=x$:
\begin{align*}
    \operatorname{CATE}_{\mathcal{R}}(x) &:= \E[Y_1-Y_0 \mid X=x]\\
    &= \E[(\gamma x_1 - x_2 + \beta + U_Y) - (\gamma x_1 - x_2 + U_Y) \mid X=x]\\
    &= \beta,
\end{align*}
and
\begin{align*}
    \operatorname{CATE}_{\mathcal{R}_{\M}}(x) &:= \E[Y_{T=1}-Y_{T=0} \mid X=x]\\
    &= \E[(\gamma(X_{T=1})_1 - (X_{T=1})_2 + \beta + U_Y) - (\gamma(X_{T=0})_1 - (X_{T=0})_2 + U_Y) \mid X=x]\\
    &= \E[\gamma(X_{T=1} - X_{T=0})_1 - (X_{T=1} - X_{T=0})_2 \mid X=x] + \beta\\
    &= \E[0 - (X_{T=1} - X_{T=0})_2 \mid X=x] + \beta\\
    &= -\E[(\alpha + U_2) - U_2 \mid X=x] + \beta\\
    &= -\alpha + \beta.
\end{align*}
In this configuration as well $\operatorname{CATE}_{\mathcal{R}}$ recovers the direct effect of $T$ in $\M$, while $\operatorname{CATE}_{\mathcal{R}_{\M}}$ recovers the total effect. This supports that $\operatorname{CATE}_{\mathcal{R}}$ is a direct effect no matter the graph from Figure~\ref{fig:decisive}: it comes from the counterfactual interpretation of fixing everything but the the treatment (Section~\ref{sec:concept}). Note that compared to the fairness illustration, there is a minus before $\alpha$ in the total effect due to $X_2$ having a negative impact on $Y$. Additionally, both effects ignore the noncausal path from $T$ to $Y$ via $X_1$ quantified by $\gamma$. Most importantly, $\operatorname{CATE}_{\mathcal{R}} \neq \operatorname{CATE}_{\mathcal{R}_{\M}}$ if $\alpha \neq 0$. As expected from the discussion in Section~\ref{sec:nonequi}, an RCM satisfying the fundamental assumptions of causal inference is generally not single-world equivalent to the latent SCM when both pre and posttreatment covariates coexist.

\begin{figure}[t]
    \centering
    \begin{subfigure}[b]{0.45\textwidth}
        \centering
        \begin{tikzpicture}[-latex ,
        state/.style ={rectangle ,top color =white, text width=1.3cm, align=center,
        draw}]
        \node[state] (Xo){Blood pressure};
        \node[state] (Xi) [above =of Xo] {Gene};
        \node[state] (Y) [below right =of Xo] {Health};
        \node[state] (S) [below left =of Xo] {Smoking};
        \path (S) edge (Xo);
        \path (Xi) edge (S);
        \path (S) edge (Y);
        \path (Xi) edge (Y);
        \path (Xo) edge (Y);
        \end{tikzpicture}
        \caption{Real-world variables}
     \end{subfigure}
     \hfill
     \begin{subfigure}[b]{0.45\textwidth}
        \centering
        \begin{tikzpicture}[-latex ,
        state/.style ={circle ,top color =white,
        draw}]
        \node[state] (Xo){$X_2$};
        \node[state] (Xi) [above =of Xo] {$X_1$};
        \node[state] (Y) [below right =of Xo] {$Y$};
        \node[state] (S) [below left =of Xo] {$T$};
        \draw[->] (S) -- (Xo) node[midway,below,xshift=3pt]{$\alpha$};
        \draw[->] (Xi) -- (S) node[midway,above,xshift=-5pt]{$+1$};
        \draw[->] (S) -- (Y) node[midway,above]{$\beta$};
        \draw[->] (Xi) -- (Y) node[midway,above,xshift=3pt]{$\gamma$};
        \draw[->] (Xo) -- (Y) node[midway,below,xshift=-3pt]{-1};
        \end{tikzpicture}
        \caption{Mathematical notations}
        \label{fig:variables}
     \end{subfigure}
    \caption{Causal graph of $\M$ in Appendix~\ref{sec:middle_case}. Linear coefficients are indicated on Figure~\ref{fig:variables}. Exogenous variables are not represented but Assumption~\ref{hyp:noises} holds.}
    \label{fig:middle_case}
\end{figure}

\section{Remaining proofs}\label{sec:proofs}

\begin{proof}\textbf{of Lemma \ref{lem:docalculus}}\mbox{}
Since $\mathcal{M}$ is acyclical, there exists a topological ordering on the indices in $\I$ with respect to the parent-child relation, and therefore on the subset ${I^c}$. This means in particular that there exist some indices $j \in {I^c}$ such that $g_j$ takes only variables in $V_I$ as endogenous inputs. Starting from these indices, and recursively substituting along the topological ordering produces a measurable $f_{I^c}$ such that
\[
V_{I^c} \aseq f_{I^c}(V_{\operatorname{Endo}({I^c}) \setminus {I^c}}, U_{\operatorname{Exo}({I^c})}).
\]
Note that $\operatorname{Endo}({I^c}) \setminus I^c = \operatorname{Endo}({I^c}) \cap I \subseteq I$. Carrying out the same substitution on the intervened model $\mathcal{M}_{V_I = v_I}$ with solution $\tilde{V}$ gives
\[
\tilde{V}_{I^c} \aseq f_{I^c}(v_{\operatorname{Endo}({I^c}) \setminus {I^c}}, U_{\operatorname{Exo}({I^c})}),
\]
while by definition $\tilde{V}_I \aseq v_I$.
\end{proof}

\begin{proof}\textbf{of Lemma \ref{lm:sw}}\mbox{ }Let $t \in \T$. Positivity ensures that $\P( \cdot \mid X=x, T=t')$ is well-defined for $\mathcal{L}(X)$-almost every $x \in \R^d$ and any $t' \in \T$. Thereby, conditional ignorability entails that $\mathcal{L}(Y_t \mid X=x, T=t) = \mathcal{L}(Y_t \mid X=x, T=t')$. Moreover, consistency implies that $\mathcal{L}(Y \mid X=x, T=t) = \mathcal{L}(Y_t \mid X=x, T=t)$. All in all, for any $t' \in \T$:
\begin{equation}\label{eq:identification}
    \mathcal{L}(Y_t \mid X=x, T=t') = \mathcal{L}(Y \mid X=x, T=t).
\end{equation}
Next, we use the above expression involving \emph{conditional} distributions to obtain an expression for the \emph{joint} distribution of $(T,X,Y_t)$. Let $F \subseteq \T \times \R^d \times \R^p$ be a Borel set. We have,
\[
    \P((T,X,Y_t) \in F) = \int \P((t',x,Y_t) \in F \mid X=x,T=t') \mathrm{d}\P(X=x,T=t').
\]
Then, we define the sets $F(t',x) := \{y \in \R^p \mid (t',x,y) \in F\}$ for every $(t',x) \in \T \times \R^d$ (which are Borel sets) so that
\begin{align*}
    \P((T,X,Y_t) \in F) &= \int \P( Y_t \in F(t',x) \mid X=x,T=t') \mathrm{d}\P(X=x,T=t')\\ &= \int \P( Y \in F(t',x) \mid X=x,T=t)\mathrm{d}\P(X=x,T=t'),
\end{align*}
where the last equality follows from \eqref{eq:identification}. This concludes the proof.
\end{proof}

\begin{proof}\textbf{of Lemma \ref{lem:consistency_do}}\mbox{ }
Let $t \in \T$. By assumption, the random vector $\O := (T,X,Y)$ is the solution to an acylical SCM. We write $U_X$ and $U_Y$ the exogenous parents of respectively $X$ and $Y$. Therefore, by partitioning $\O$ into $T$ and $(X,Y)$, Lemma~\ref{lem:docalculus} guarantees the existence of a measurable function $f_{X,Y}$ such that
\begin{align*}
    (X,Y) &\aseq f_{X,Y}(T,U_X,U_Y),\\
    (X_{T=t},Y_{T=t}) &\aseq f_{X,Y}(t,U_X,U_Y).
\end{align*}
Therefore, selecting the coordinates corresponding to $Y$ furnishes a measurable function $\tilde{f}_Y$ such that
\begin{align*}
    Y &\aseq \tilde{f}_Y(T,U_X,U_Y),\\
    Y_{T=t} &\aseq \tilde{f}_Y(t,U_X,U_Y).
\end{align*}
These expressions hold on a measurable set $\Omega^* \subseteq \Omega$ such that $\P(\Omega^*) = 1$. To conclude, simply observe that for any $\omega \in \Omega^*$ such that $T(\omega)=t$, we have
\[
Y(\omega) = \tilde{f}_Y(t,U_X(\omega),U_Y(\omega)) = Y_{T=t}(\omega).
\]
\end{proof}

\begin{proof}\textbf{of Lemma \ref{lm:noise}}\mbox{ }
By assumption, the random vector $\O := (T,X,Y)$ is the solution to an acylical SCM where $U_T$, $U_X$ and $U_Y$ denote the exogenous parents of respectively $T$, $X$ and $Y$. Since additionally $Y_{\operatorname{Endo}(T)} = Y_{\operatorname{Endo}(X_i)} = \emptyset$ for every $i \in \{1,\ldots,d\}$ (Assumption~\ref{hyp:outcome}), Lemma~\ref{lem:docalculus} ensures the existence of a measurable function $f_{T,X}$ such that $(T,X) \aseq f_{T,X}(U_T,U_X).$ Therefore, if $U_Y \independent (U_T,U_X)$ (Assumption~\ref{hyp:noises}), then $U_Y \independent (T,X)$.
\end{proof}

\begin{proof}\textbf{of Proposition \ref{prop:existence}}\mbox{ }
Throughout the proof, we write $X_i$, $Y_i$, and $Y_{t,i}$ for the $i$th component of respectively $X$, $Y$, and $Y_t$ for any $t \in \T$.

We start by finding a generative model for the RCM $\mathcal{R} := (T,X,(Y_t)_{t \in \T})$ using a standard result from probability theory: any Borel probability measure on a Polish space can be obtained by \emph{push-forward} of the uniform measure on $[0,1]$ denoted by $\operatorname{Unif}$. This follows from the existence of a measurable bijection with measurable inverse between any Polish space and $\R$ \citep[Theorem 15.6]{kechris2012classical}, and the fact that any Borel probability measure on $\R$ can always be obtained by push-forward of $\operatorname{Unif}$ by its generalized-inverse probability distribution function. As a consequence, there always exists a measurable function $\psi : [0,1] \to \T \times \R^d \times (\R^p)^{N+1}$ such that $\law{(T,X,(Y_t)_{t \in \T})} = \operatorname{Unif} \circ \psi^{-1}$. Thereby, there exists a random variable $U$ such that $\law{U} = \operatorname{Unif}$ and $(T,X,(Y_t)_{t \in \T}) \aseq \psi(U)$. We divide $\psi$ into $\psi = (\psi_T,(\psi_{X_i})^d_{i=1},(\psi_{Y_{0,i}})^p_{i=1},\ldots,(\psi_{Y_{N,i}})^p_{i=1})$ and define $\psi_X := (\psi_{X_i})^d_{i=1}$ and $\psi_{Y_t} := (\psi_{Y_{t,i}})^p_{i=1}$ for any $t \in \T$. This enables us to write $T \aseq \psi_T(U)$, $X \aseq \psi_X(U)$, and $Y_t \aseq \psi_{Y_t}(U)$ for any $t \in \T$.

Then, on the basis of $\psi$ and $U$, we construct an SCM $\M := \langle U, g \rangle$ that satisfies two properties: (1) $\M$ is acyclical with solution $(T,X,Y)$, (2) $Y_{T=t} \aseq Y_t$ for any $t \in \T$. We define $g := (g_T,(g_{X_i})^d_{i=1},(g_{Y_i})^p_{i=1})$ where $g_T := \psi_T$, $g_{X_i} = \psi_{X_i}$, and $g_{Y_i}$ is given by $(t,u) \mapsto \psi_{Y_{t,i}}(u)$. The associated structural equations correctly define $T$ and $X$ by definition of $\psi_T$ and $\psi_X$. They also correctly define $Y$ since the fact that $(Y_t)_{t \in \T}$ satisfies the consistency rule means that $Y \aseq \sum_{t \in \T} \mathbf{1}_{\{T=t\}} \psi_{Y_t}(U) = \sum_{t \in \T} \mathbf{1}_{\{T=t\}} g_Y(t,U) = g_Y(T,U)$ where $g_Y := (g_{Y_i})^p_{i=1}$. Additionally, $g$ clearly does not represent cyclic relationships between variables.
Therefore, $\M := \langle U, g \rangle$ belongs to $\mathfrak{M}_\O$. Moreover, do-interventions in this model produces for any $t \in \T$, $Y_{T=t} \aseq g_{Y}(t,U) = \psi_{Y_t}(U) \aseq Y_t$. Therefore, $\mathcal{R}_\M$ and $\mathcal{R}$ are almost-surely equivalent. This concludes the proof.
\end{proof}

\bibliographystyle{abbrvnat}
\bibliography{references}

\end{document}